\documentclass[11pt]{amsart}

\usepackage[USenglish]{babel}

\usepackage{amsmath,amsthm,amssymb,amscd}

\usepackage{booktabs}
\usepackage[mathscr]{eucal}
\usepackage[normalem]{ulem}
\usepackage{latexsym,youngtab}
\usepackage{multirow}
\usepackage{epsfig}
\usepackage{parskip}
\usepackage{enumerate}
\usepackage{tikz}

\newcommand{\udots}{\rotatebox{85}{$\ddots$}}

\usepackage[margin=2cm]{geometry}
\author[A. Conner, F. Gesmundo, J. M. Landsberg, E. Ventura, and Y. Wang]{Austin Conner, Fulvio Gesmundo, Joseph M. Landsberg,\\ Emanuele Ventura, Yao Wang}

\address[A. Conner, J. M. Landsberg, Y. Wang]{Department of Mathematics, Texas A\&M University, College Station, TX 77843-3368, USA}
\email[A. Conner]{connerad@math.tamu.edu}
\email[J. M. Landsberg]{jml@math.tamu.edu}
\email[Y. Wang]{wangyao@math.tamu.edu}

\address[F. Gesmundo]{QMATH, Dept. of Mathematical Sciences, U. of Copenhagen, Universitetsparken 5, 2100 Copenhagen O., Denmark}
\email{fulges@math.ku.dk}

\address[E. Ventura]{Mathematisches Institut, Universit\"at Bern, Sidlerstrasse 5, 3012 Bern, Switzerland}
\email[E. Ventura]{emanuele.ventura@math.unibe.ch}

\title[Geometry of the Asymptotic Rank Conjecture]{Towards a Geometric Approach to Strassen's Asymptotic Rank Conjecture}

\thanks{Landsberg is supported by NSF grants DMS-1405348 and AF-1814254 and by 
the grant 346300 for IMPAN from the Simons Foundation and the matching 2015-2019 
Polish MNiSW fund as well as a Simons Visiting Professor grant supplied by the 
Simons Foundation and by the Mathematisches Forschungsinstitut Oberwolfach. 
Gesmundo acknowledges financial support from the European Research Council (ERC 
Grant Agreement no. 337603), the VILLUM FONDEN via the QMATH Centre of 
Excellence (Grant no. 10059) and the ICERM program in Nonlinear Algebra in Fall 
2018 (NSF DMS-1439786). Ventura acknowledges financial support by the grant 
346300 for IMPAN from the Simons Foundation and the matching 2015-2019 Polish 
MNiSW fund.}

\keywords{Tensor rank, Asymptotic rank, Matrix multiplication complexity, Slice 
rank}

\subjclass[2010]{15A69; 14L35, 68Q15}

\newcommand{\tincompr}{\operatorname{incompr}} 
\renewcommand{\th}{\theta}

\newcommand{\slrk}{\operatorname{slrk}}

\newcommand{\ot}{\otimes}
\newcommand{\op}{\oplus}

\newcommand{\BC}{\mathbb{C}}
\newcommand{\BP}{\mathbb{P}}

\newcommand{\vvirg}{, \dots ,}

\newcommand{\Hom}{\mathrm{Hom}}

\newcommand{\aaa}{\mathbf{a}}
\newcommand{\bbb}{\mathbf{b}}
\newcommand{\ccc}{\mathbf{c}}
\newcommand{\nnn}{\mathbf{n}}

\newcommand{\bbC}{\mathbb{C}}

\newcommand{\bbN}{\mathbb{N}}
\newcommand{\bbP}{\mathbb{P}}
\newcommand{\bbR}{\mathbb{R}}

\newcommand{\bbZ}{\mathbb{Z}}

\newcommand{\bfa}{\mathbf{a}}
\newcommand{\bfb}{\mathbf{b}}
\newcommand{\bfc}{\mathbf{c}}
\newcommand{\bfn}{\mathbf{n}}

\newcommand{\calA}{\mathcal{A}}

\newcommand{\cF}{\mathcal{F}}

\newcommand{\calS}{\mathcal{S}}
\newcommand{\cS}{\mathcal{S}}
\newcommand{\calT}{\mathcal{T}}

\newcommand{\aR}{\uwave{\mathbf{R}}}
\newcommand{\bfR}{\mathbf{R}}
\newcommand{\ur}{\underline{\mathbf{R}}}
\newcommand{\uR}{\underline{\mathbf{R}}}

\newcommand{\aQ}{\uwave{\mathbf{Q}}}

\newcommand{\uQ}{\underline{\mathbf{Q}}}

\newcommand{\Id}{\mathrm{Id}}

\newcommand{\Mone}{M_{\langle 1 \rangle}}
\newcommand{\Mn}{M_{\langle \bfn \rangle}}
\newcommand{\Mamu}[1]{M_{\langle #1 \rangle}}

\newcommand{\fraksl}{\mathfrak{sl}}

\newcommand{\frakS}{\mathfrak{S}}

\newcommand{\frakgl}{\mathfrak{gl}}

\newcommand{\frakg}{\mathfrak{g}}

\newcommand{\frakz}{\mathfrak{z}}

\newcommand{\frakJ}{\mathfrak{J}}

\newcommand{\Tight}{\mathit{Tight}}
\newcommand{\MaMu}{\mathit{MaMu}}
\newcommand{\Oblique}{\mathit{Oblique}}
\newcommand{\Free}{\mathit{Free}}

\renewcommand{\bar}[1]{\overline{#1}}

\renewcommand{\hat}[1]{\widehat{#1}}

\renewcommand{\emptyset}{\font\cmsy = cmsy11 at 11pt
 \hbox{\cmsy \char 59}
}

\newcommand{\dotitem}{\item[$\cdot$]}

\newcommand{\trace}{\mathrm{trace}}

\newtheorem{theorem}{Theorem}[section]
\newtheorem{proposition}[theorem]{Proposition}
\newtheorem{lemma}[theorem]{Lemma}

\newtheorem{conjecture}[theorem]{Conjecture}

\theoremstyle{definition}
\newtheorem{definition}[theorem]{Definition}
\newtheorem{example}[theorem]{Example}
\newtheorem{problem}[theorem]{Problem}
\newtheorem{question}[theorem]{Question}

\theoremstyle{remark}
\newtheorem{remark}[theorem]{Remark}

\newcommand{\Mat}{\mathrm{Mat}}

\newcommand{\supp}{\mathrm{supp}}

\begin{document}

\begin{abstract}
We make a first geometric study of three varieties in $\BC^m\ot \BC^m\ot \BC^m$ (for each $m$), including the Zariski closure of the set of {\it tight  tensors}, the tensors with continuous regular symmetry.  Our motivation is to develop a geometric framework for Strassen's {\it asymptotic rank conjecture} that the asymptotic rank of any tight tensor is minimal. In particular, we determine the dimension of the set of tight tensors. We prove that this dimension equals the dimension of the set of {\it oblique tensors}, a less restrictive class introduced by Strassen. 
\end{abstract}

\maketitle

\section{Introduction}\label{intro}

Strassen's asymptotic rank conjecture (Conjecture \ref{strconj} below) is a generalization of the famous conjecture that the exponent of matrix multiplication is two. An even further generalization of it is posed as a question in \cite[Problem 15.5]{BCS}. The object of Strassen's conjecture is the class of \emph{tight tensors}, originally introduced because of combinatorial properties that make tight tensors useful for Strassen's laser method for proving upper bounds on the exponent of matrix multiplication. Strassen proved remarkable properties about this class of tensors that led to the conjecture. 

The purpose of this paper is to place Strassen's asymptotic rank conjecture, together with its generalizations to less restrictive classes of tensors, in a geometric framework as a first step to comparing them and developing approaches for attacking them with geometric methods.
 
We make a first geometric study of algebraic varieties defined by three classes of tensors, each characterized by combinatorial properties. These classes arise in algebraic complexity theory \cite{Strassen:AlgebraComplexity}, quantum information theory \cite{2017arXiv170907851C}, and geometric invariant theory (more precisely, the study of moment polytopes \cite{MR932055,MR765581,MR1923785}). We pose questions intermediate to the asymptotic rank conjecture and  Problem 15.5 in \cite{BCS} for these classes of tensors. We compare these varieties with the well-studied orbit closure of the matrix multiplication tensor and the ambient space.

\subsection{Definitions and Notation}
Throughout the paper, $A,B,C$ denote complex vector spaces respectively  of dimension $\bfa,\bfb,\bfc$. Given two tensors $T_1 \in A_1 \otimes B_1 \otimes C_1$ and $T_2 \in A_2 \otimes B_2 \otimes C_2$, one can regard the tensor $T_1 \otimes T_2$ as an element of $(A_1 \otimes A_2) \otimes (B_1 \otimes B_2) \otimes (C_1 \otimes C_2)$. This is called \emph{Kronecker product} of $T_1$ and $T_2$ and it is denoted by $T_1 \boxtimes T_2$. Kronecker powers are defined iteratively: for $T \in A \otimes B \otimes C$, let $T^{\boxtimes N} = T^{\boxtimes N-1} \boxtimes T$, which is a tensor in $(A^{\otimes N}) \otimes (B^{\otimes N}) \otimes (C^{\otimes N})$. The direct sum of $T_1$ and $T_2$ is the tensor $T_1 \oplus T_2 \in (A_1 \oplus A_2) \otimes (B_1 \oplus B_2) \otimes (C_1 \oplus C_2)$.

A tensor $T \in A \otimes B \otimes C$ is {\it concise} if the three linear maps $T_A : A^* \to B \otimes C$, $T_B: B^* \to A \otimes C$, and $T_C : C^* \to A \otimes B$ are injective. Kronecker products of concise tensors are concise, and, in particular, if $T$ is concise then $T^{\boxtimes N}$ is concise as well. 

The \emph{rank} of $T \in A\ot B\ot C$, denoted $\bfR(T)$, is the smallest  integer  $r$ such that $T=\sum_{j=1}^r u_j\ot v_j\ot w_j$ with $u_j\in A$, $v_j\in B$, $w_j\in C$. The {\it border rank} of $T$,  denoted  $\uR(T)$, is the smallest $r$ such that $T$ may be expressed as a limit (in the Euclidean topology) of tensors of rank $r$. The \emph{asymptotic rank} of $T$ is $\aR(T) = \lim_{N \to \infty} \bfR(T^{\boxtimes N})^{1/N} =  \lim_{N \to \infty} \uR(T^{\boxtimes N})^{1/N}$. In \cite{MR882307}, these limits are shown to exist and to be equal.

For every tensor $T \in A \otimes B \otimes C$, we have $\bfR(T) \geq \uR(T) \geq \aR(T)$; if $T$ is concise then $\aR(T) \geq \max\{\bfa,\bfb,\bfc\}$. When equality holds we say $T$ has {\it minimal} asymptotic rank. Moreover, $\bfR(T_1 \boxtimes T_2) \leq \bfR(T_1) \bfR(T_2)$, and similarly for border rank and asymptotic rank.

Border rank and asymptotic rank are lower semicontinuous under degeneration. More precisely, let $G := GL(A) \times GL(B) \times GL(C)$ and let $T,T' \in A \otimes B \otimes C$. We say that $T'$ is a {\it degeneration}  of $T$ if $T' \in \bar{G\cdot T}$, where $\bar{G\cdot T}$ denotes the orbit closure (equivalently in Zariski or Euclidean topology) of the tensor $T$ under the natural action of $G$. One has  $\uR(T') \leq \uR(T)$ and $\aR(T') \leq \aR(T)$.

Let $\Mone$ denote any rank one tensor. After possibly re-embedding $A,B,C$ in larger dimensional spaces, the border rank of a tensor $T \in A \otimes B \otimes C$ is characterized as the smallest $r$ such that $\Mone^{\oplus r}$ degenerates to $T$.

Given $m \in \bbN$, let $[m] := \{ 1 \vvirg m  \}$. For a subset $\calS \subseteq [\bfa] \times [\bfb] \times [\bfc]$, let $|\calS|$ denote its cardinality. Given a tensor $T = \sum_{ijk} T^{ijk} a_i \otimes b_j \otimes c_k$ with $\{ a_i \}$ a basis of $A$ and similarly for $\{b_j\}$ and $\{c_k\}$, the \emph{support} of $T$ in this basis is the set $\supp(T) = \{ (i,j,k) : T^{ijk} \neq 0\} \subseteq [\bfa]\times [\bfb]\times [\bfc]$. We say that a set $\calS \subseteq [\bfa] \times [\bfb] \times [\bfc]$ is concise if the restrictions of the three projections on $[\bfa]$, $[\bfb]$ and $[\bfc]$ to 
$\calS$  are surjective. Generic tensors with concise support are concise.

From a geometric perspective, tightness is a property concerning the stabilizer of $T$ under the action of $G$: a tensor is \emph{tight} if the stabilizer of $T$ in $G$ contains a regular semisimple element. The computer science literature (see, e.g., \cite{BCS,BlaserNotes}) generally works with an equivalent combinatorial definition in terms of the support of $T$ in a preferred basis. We refer to \cite{Strassen:AlgebraComplexity} and Section \ref{subsec: tight tensors} for details on the geometric definition and the 
proof of the equivalence between the two definitions. The combinatorial point of view naturally offers two generalizations,  which already appeared in \cite{MR882307}.

\begin{definition}\label{defin: tight oblique and free}
 A concise subset $\calS \subseteq [\bfa] \times [\bfb]\times [\bfc]$ is called
\begin{itemize}
 \dotitem \emph{tight} if there exist injective functions $\tau_A : [\bfa] \to \bbZ$, $\tau_B : [\bfb] \to \bbZ$ and $\tau_C : [\bfc] \to \bbZ$ such that $\tau_A(i) + \tau_B(j) + \tau_C(k)=0$ for every $(i,j,k) \in \calS$.

 \dotitem \emph{oblique} if no two elements of $\calS$  are comparable  under the partial ordering on $[\bfa] \times [\bfb] \times [\bfc]$ induced by total orders on $[\bfa]$,$[\bfb]$,$[\bfc]$ (one says $\calS$ is an {\it antichain} in the partially ordered set $[\bfa]\times [\bfb]\times [\bfc]$);
 
 \dotitem \emph{free} if any two $(i_1,j_1,k_1),(i_2,j_2,k_2) \in \calS$ differ in at least two entries.
 \end{itemize}
A tensor $T \in A \otimes B \otimes C$ is \emph{tight} (resp. \emph{oblique}, resp. \emph{free}) if there exists a choice of bases $\{ a_i\}_{i \in [\bfa]}, \{ b_j\}_{j \in [\bfb]},\{ c_i\}_{k \in [\bfc]}$ such that the support $\calS \subseteq [\bfa] \times [\bfb] \times [\bfc]$ of $T$ in the given bases is a tight (resp. oblique, resp. free) subset. 
In this case, the chosen basis is called a tight (resp. oblique, resp. free) basis.
\end{definition}

Every tight tensor is oblique and   every oblique tensor is free, see Remarks \ref{rmk: tight is oblique} and \ref{rmk: oblique is free}. We are unaware of geometric definitions of obliqueness and freeness.

\begin{problem}
 Find geometric characterizations for obliqueness and freeness.
\end{problem}

\subsection{A hierarchy of Conjectures on the asymptotic rank}
The matrix multiplication tensor $\Mn \in \Mat_{\bfn}^* \otimes \Mat_{\bfn}^* \otimes \Mat_{\bfn}$ is the bilinear map sending two matrices of size $\bfn \times \bfn$ to their product. This notation is consistent with the one used for rank one tensors: indeed, $\Mone \in \bbC^1 \otimes \bbC^1 \otimes \bbC^1$ may be regarded as the scalar multiplication bilinear map. Matrix multiplication has the self-reproducing property $\Mn^{\boxtimes N} = \Mamu{\bfn^N}$. Moreover, $\Mamu{\bfn}$ is tight. The famous conjecture that the exponent of matrix multiplication is two may be phrased in terms of the asymptotic rank:

\begin{conjecture}\label{conj: mamu} 
For some (and as a consequence all) $\nnn>1$, $\aR(\Mn)=\bfn^2$, i.e., $\Mn$ has minimal asymptotic rank.
\end{conjecture}

Strassen's asymptotic rank conjecture generalizes Conjecture \ref{conj: mamu} to the class of tight tensors. It was stated in a more general form in \cite{Strassen:AlgebraComplexity} (see also Conjecture \ref{conj: strassen functionals}), but it appeared implicitly already in \cite{Str:AsySpectrumTensors,MR1089800}.

\begin{conjecture}[Strassen's Asymptotic Rank Conjecture, \cite{Strassen:AlgebraComplexity}]\label{strconj}
 Let $T \in \BC^m\ot \BC^m\ot \BC^m$ be tight and concise. Then $\aR(T) = m$, i.e., all concise tight tensors have minimal asymptotic rank.
\end{conjecture}

One can consider generalizations of Conjecture \ref{strconj} to the class of oblique and free tensors. These generalizations are natural in light of Strassen's work: they implicitly already appeared in \cite{Str:AsySpectrumTensors,MR1089800,Strassen:AlgebraComplexity} and later in \cite{BCS,2017arXiv170907851C}. We pose the following questions to make them explicit.

\begin{question}\label{conj: oblique}
Let $T \in \bbC^m \otimes \bbC^m \otimes \bbC^m$ be oblique and concise. Is $\aR(T) = m$? In other words, do all concise oblique tensors have minimal asymptotic rank? 
\end{question}

\begin{question}\label{conj: free}
Let $T \in \bbC^m \otimes \bbC^m \otimes \bbC^m$ be free and concise. Is $\aR(T) = m$? In other words, do all concise free tensors have minimal asymptotic rank? 
\end{question}

In \cite{BCS}, the authors posed the question in full generality:
\begin{question}[\cite{BCS}, Problem 15.5] \label{bcsconj}
Is $\aR(T)=m$ for all concise $T\in \BC^m\ot \BC^m\ot \BC^m$? In other words, do all tensors have minimal asymptotic rank? 
\end{question}

This gives a hierarchy of affirmative answers to each of the five problems stated above: 

\begin{center}
Question \ref{bcsconj}
$\Rightarrow$
  Question \ref{conj: free} $\Rightarrow$ 
    Question \ref{conj: oblique} $\Rightarrow$ 
      Conjecture \ref{strconj} $\Rightarrow$ Conjecture \ref{conj: mamu}. 
\end{center}
 
Our goal is to study to what extent these problems are different.  As a step towards this goal, we compare the sets defined by each of the problems. In particular, we determine the dimensions of the varieties of tensors to which each of the five problems applies.

Let $\bar{\Tight}_m$, $\bar{\Oblique}_m$ and $\bar{\Free}_m$ be the closures (equivalently in Zariski or Euclidean topology) of the sets of concise tight, oblique and free tensors respectively. Let 
  \[
\bar{\MaMu}_m := \bar{GL(A)\times GL(B)\times GL(C) \cdot \Mn} \subseteq A 
\otimes B \otimes C   
  \]
  with $\bfa = \bfb = \bfc = m = \bfn^2$. Then Conjecture \ref{conj: mamu} may  be rephrased as follows. If  $T \in \bar{\MaMu}_m$, then $\aR(T) \leq m$. Similar reformulations  of the other questions/conjectures can be given in terms of the varieties $\bar{\Tight}_m,\bar{\Oblique}_m,\bar{\Free}_m$ and finally in terms of the space  $\bbC^m \otimes \bbC^m \otimes \bbC^m$.

 The following result determines the dimensions of the varieties $\bar{\MaMu}_m$, $\bar{\Tight}_m$, $\bar{\Oblique}_m$, $\bar{\Free}_m$, providing a first comparison among the sets of tensors to which each of the conjectures mentioned above applies.

\begin{theorem}\label{thm: dimension of classes of tensors}
 Let $m \geq 2$ and let $\bfa = \bfb = \bfc = m$. Then
\begin{enumerate}[(i)]
\item if $m = \bfn^2$, then $\dim \bar{\MaMu}_m = 3m^2 - 3m$;
\item $\dim \bar{\Tight}_m = 3m^2 + \lceil \frac{3}{4} m^2 \rceil - 3m$;
\item $\dim \bar{\Oblique}_m = 3m^2 + \lceil \frac{3}{4} m^2 \rceil - 3m$;
\item $\dim \bar{\Free}_m = 4m^2 - 3m$.
\end{enumerate}
\end{theorem}

The statement of (i) dates back at least to \cite{dGro:VarsOptAlgIIsoGrps}. The proofs of the remaining statements are obtained in Section \ref{tightsect} by applying a natural vector bundle construction (an incidence correspondence) to the explicit maximal supports for each case in Theorem \ref{thm: maximum support} below, which is also proved in Section \ref{tightsect}.

Theorem \ref{thm: dimension of classes of tensors} shows that for the varieties of interest for all the problems stated above, except Question \ref{bcsconj}, the dimension is quadratic in the dimension $m$. We were surprised to discover that the varieties $\bar{\Tight}_m $ and $\bar{\Oblique}_m$ have the same dimension. This can be viewed as suggesting that the different forms of the asymptotic rank conjecture are similar, at least in terms of the set of tensors to which they apply.

One can study asymptotic rank for any single tensor which does not have minimal border rank. Because of the self-reproducing property of the matrix multiplication tensor, proving $\aR(M_{\langle 2 \rangle}) = 4$ (or $\aR(M_{\langle \bfn \rangle}) = \bfn^2$ for any single $\bfn \geq 2$) would prove Conjecture \ref{conj: mamu}, and thus that the exponent of matrix multiplication is two. Even more interestingly, a consequence of \cite{CopperWinog:MatrixMultiplicationArithmeticProgressions} (see \cite[Remark 15.44]{BCS}) is that Conjecture \ref{conj: mamu} would follow were $\aR(T_{cw,2}) = 3$, where $T_{cw,2} \in \bbC^3 \otimes \bbC^3 \otimes \bbC^3 $ is the small Coppersmith-Winograd tensor. In \cite{CGLVkron}, we observed that Conjecture \ref{conj: mamu} would follow from other explicit tensors having minimal asymptotic rank. 

More generally, Conjecture \ref{conj: mamu} would follow from the $m=3$ case of any of the generalizations; in the case $m=3$, the four generalizations reduce to two:

\begin{theorem}\label{m3thm}
 Let $\bfa = \bfb = \bfc = 3$. Then $A \otimes B \otimes C = \bar{\Free}_3$ and $\bar{\Tight}_3=\bar{Oblique}_3$, which has codimension $2$.
\end{theorem}

In addition, in \cite{CGLVkron} we provided numerical evidence showing that if $T$ is generic in $\bbC^3 \otimes \bbC^3\otimes \bbC^3$, then $\uR(T^{\boxtimes 2}) \leq 22 < 5^2 = \uR(T)$. This can be taken as positive evidence for the generalization proposed in Question \ref{bcsconj} in the case $m=3$. On the other hand, we proved that for $m > 3$, there are tensors such that $\uR(T) = m+1$ is not minimal and $\uR(T^{\boxtimes 2} ) = (m+1)^2$ and $\uR(T^{\boxtimes 3} ) = (m+1)^3$; in particular, $T_{cw,q}$ shows this multiplicative behaviour and this can be taken as negative evidence even for Conjecture \ref{strconj} for general $m$.

When  $\bfa =\bfb = \bfc = 4$, $\bar{\Free_4}$ is a variety of codimension $2$, $\bar{\Oblique}_4$ and $\bar{\Tight}_4$ have codimension $16$, and $\MaMu_4$ has codimension $28$. Proposition \ref{prop: tight strict in oblique} shows that the inclusion $\bar{\Tight}_m \subseteq \bar{\Oblique}_m$ is strict for all $m \geq 4$.

The following result determines the maximum possible support of a tight, oblique or free tensor in a tight, oblique or free basis.
\begin{theorem}\label{thm: maximum support}
 Let $\calS \subseteq [m] \times [m] \times [m]$. Then
 \begin{enumerate}[(i)]
  \item if $\calS$ is tight then $\vert \calS \vert \leq \lceil \frac{3}{4} m^2 
\rceil$ and the inequality is sharp;
  \item if $\calS$ is oblique then $\vert \calS \vert \leq \lceil \frac{3}{4} 
m^2 \rceil$ and the inequality is sharp;
  \item if $\calS$ is free then $\vert \calS \vert \leq m^2 $ and the inequality 
is sharp.  
 \end{enumerate}
\end{theorem}
Note that  with  $m = \bfn^2$, the standard presentation of the matrix multiplication tensor gives $| \supp(\Mn)| = \lfloor \frac{3}{2}m^2 \rfloor$. 
  
The sharpness results in Theorem \ref{thm: maximum support} follow by exhibiting explicit supports with the desired cardinality. The support described in the proof of Theorem \ref{thm: maximum support}(i) is used in \cite{LMhighbr} to construct the first explicit sequence (depending on $m$) of tensors in $\BC^m\ot \BC^m\ot \BC^m$ of border rank greater than $2m$.

In Section \ref{section: compressibility}, we discuss {\it compressibility} and {\it slice rank} of tensors. We explain {\it Strassen's support functionals} and how they motivate Conjecture \ref{strconj}. In particular, we prove that tight tensors are far more compressible than generic tensors (Theorem \ref{compressthm}), which could be taken as evidence to favor Conjecture \ref{strconj} over the other more general problems.

In Section \ref{section: propagation of symmetries}, we establish results on the growth of symmetry groups of tensors under direct sums and Kronecker products; see Theorem \ref{thm: propagation of symmetries}. The dimension of the symmetry group of a tensor is a geometric invariant which is upper semicontinuous under degeneration. In particular, the result of Theorem \ref{thm: propagation of symmetries}(iii) shows that tensors which are generic in terms of dimension of symmetry group (namely having a $0$-dimensional symmetry group), remain generic under Kronecker product. This can also be interpreted as evidence to favor Conjecture \ref{strconj} over Questions \ref{conj: oblique}, \ref{conj: free} and \ref{bcsconj}.

Finally, Theorem \ref{thm: propagation of symmetries} is motivated by the connection between symmetries of a tensor and Strassen's laser method. We refer to \cite{CGLVkron} for details on the method: we mention here that this technique can be applied to \emph{block tight tensors}, defined implicitly in \cite[\S 15.6]{BCS} and explicitly  in \cite[Def. 5.1.4.2]{LCBMS}, a property weaker than tightness but still implying the tensor has  continuous symmetries.

\subsection*{Acknowledgements}
We thank the anonymous referees for their valuable comments on an earlier version of this paper. 

\section{Tight, oblique, and free tensors}\label{tightsect}

In this section, we establish information about the sets of tight, oblique and 
free tensors, and prove Theorems \ref{thm: dimension of classes of tensors} and 
\ref{thm: maximum support}.

\subsection{Tight tensors}\label{subsec: tight tensors}

Tightness can be characterized in terms of the stabilizer of a tensor in $A 
\otimes B \otimes C$ under the action of $G= GL(A) \times GL(B) \times GL(C)$. 
We introduce some useful notation and definitions.

Let $\Phi : GL(A) \times GL(B) \times GL(C) \to GL(A \otimes B \otimes C)$ be 
the group homomorphism defining the natural action of $GL(A) \times GL(B) \times 
GL(C)$ on $A \otimes B \otimes C$; $\Phi$ has a $2$-dimensional kernel 
$Z_{A,B,C} = \{ (\lambda \Id_A, \mu \Id_B, \nu \Id_C) : \lambda\mu\nu = 1\}$, so 
that $G := (GL(A) \times GL(B) \times GL(C)) / Z_{A,B,C}$ can be regarded as a 
subgroup of $GL(A \otimes B \otimes C)$. The symmetry group of $T$, denoted 
$G_T$, is the stabilizer in $G$ under this action, that is $G_T := \{ g \in G: 
g\cdot T = T\}$.

The differential $d\Phi$ of $\Phi$ induces a map at the level of Lie algebras: 
write $\frakg_T$ for the annihilator of a tensor $T$ under the action of 
$(\frakgl(A) \oplus \frakgl(B) \oplus \frakgl(C)) / \frakz_{A,B,C}$ where 
$\frakz_{A,B,C}\simeq \bbC^2$ is the Lie algebra of $Z_{A,B,C}$: explicitly 
$\frakz_{A,B,C} = \{ (\lambda \Id_A ,\mu\Id_B ,\nu \Id_C : \lambda + \mu + \nu = 
0\}$. Since $\frakg_T$ is the Lie algebra of $G_T$, it determines the continuous 
symmetries of $T$, i.e., the connected component of the identity of $G_T$.

Fix $T \in A \otimes B \otimes C$. Then $T$ is tight if and only if $\frakg_{T}$ 
contains a regular semisimple element of $(\frakgl(A) \oplus \frakgl(B) \oplus 
\frakgl(C) )/ \frakz_{A,B,C}$. A regular semisimple element is a triple $L = 
(X,Y,Z)$ which, under some choice of bases, is represented by diagonal matrices 
$X,Y,Z$, each of them having distinct eigenvalues. Equivalently, $T$ 
is stabilized by a regular semisimple one-parameter subgroup of $(GL(A) \times 
GL(B) \times GL(C)) / Z_{A,B,C}$. Observe that the tightness of $T$ in a given 
basis only depends on the support of $T$; in particular, the eigenvalues of the 
three matrices in $L = (X,Y,Z)$, suitably rescaled, provide the functions 
$\tau_A,\tau_B,\tau_C$ of Definition \ref{defin: tight oblique and free}. We 
refer to \cite{MR1089800,MR2138544} for the complete proof that the two 
characterizations are equivalent.

\begin{example}[A tight support of cardinality $\lceil \frac{3}{4} m^2\rceil$] 
\label{example: max tight support}
 Let $m \geq 0$ be an odd integer and write $m = 2\ell +1$. Define 
 \[
  \calS_{t\text{-}max, m} = \left\{  (i,j,k) \in [m] \times [m] \times [m]: i + 
j + k = 3\ell \right\}.
 \]
By Definition \ref{defin: tight oblique and free}, $\calS_{t\text{-}max,m}$ is 
tight. Let $\bfa=\bfb=\bfc=m$ and let $T \in A \otimes B \otimes C$ be any 
tensor with support $\calS_{t\text{-}max,m}$. Let $L = (U,V,W) \in \frakgl(A) 
\oplus \frakgl(B) \oplus \frakgl(C)$ be the triple of diagonal matrices $U = V = 
W$ having $i - \ell$ at the $i$-th diagonal entry, with $i = 0 \vvirg m-1$. Then 
$L . T = 0$, because for every element $(i,j,k) \in \supp(T)$ we have 
\[
L . (a_i \otimes b_j \otimes c_k) = [(i-\ell) + (j-\ell) + (k-\ell)] a_i \otimes 
b_j \otimes c_k = 0.
\]
If $T$ has support $\calS_{t\text{-}max,m}$, one can write $T = \sum_{jk} T^{jk} 
a_{3\ell - j - k} \otimes b_j \otimes c_k$. We can represent $T$ as an $m \times 
m$ matrix whose entries are elements of $A$; in this case, we have
\begin{equation}\label{eqn: matrix max support}
\left[
\begin{array}{ccccccc}
&  & &T^{0,\ell} a_{2\ell} & \cdots & T^{0,2\ell-1} a_{\ell+1} & T^{0,2\ell} 
a_\ell  \\[2em]
&  & \udots & & &  \udots & T^{1,2\ell} a_{\ell-1} \\[2em] & \udots & & &  
\udots & & \vdots  \\[2em]
T^{\ell,0} a_{2\ell}  & & & \udots & & &  T^{\ell,2\ell} a_{0} \\[2em]
 \vdots &  & \udots & & & \udots  &  \\[2em]
T^{2\ell-1,0} a_{\ell+1} & \udots & & & \udots &  &   \\[2em] 
T^{2\ell,0} a_\ell  & T^{2\ell,1} a_{\ell -1}& \cdots & T^{2\ell , \ell} a_{0} & 
& &
\end{array}
\right].
\end{equation}
Each nonzero entry in this matrix corresponds to an element of 
$\calS_{t\text{-}max,m}$; each of the two triangles of $0$'s (the top left and 
the bottom right) consists of $\binom{\ell+1}{2}$ entries. Therefore the number 
of nonzero entries is $(2\ell+1)^2 - (\ell+1)(\ell) = 3\ell^2 + 3\ell +1 = 
\lceil \frac{3}{4}m^2 \rceil$.

If $m = 2\ell$ is even, one obtains a tight support of cardinality $3\ell^2 = 
\lceil \frac{3}{4}m^2 \rceil$ by erasing the last row and the last column of the 
matrix and setting $a_0$ to $0$. Geometrically this is equivalent to applying 
the projection which sends $a_0,b_{2\ell},c_{2\ell}$ to $0$ and the other basis 
vectors of the odd dimensional spaces to basis vectors of the even dimensional 
spaces. Explicitly, if one has bases $\{ a_0 \vvirg a_{2\ell-1}\}$, $\{ b_0 
\vvirg b_{2\ell-1}\}$,$\{ c_0 \vvirg c_{2\ell-1}\}$ of the spaces $A,B,C$ of 
dimension $2\ell$, the tight support is determined by the functions $\tau_A (i) 
= i-\ell +1$, $\tau_B (j) = \tau_C(j) = j-\ell$.
\end{example}

It turns out that the element $L$ introduced in Example \ref{example: max tight 
support} is, up to scale, the only non-trivial element of $\frakg$ which 
annihilates  a generic tensor with support $\calS_{t\text{-}max,m}$, as shown in 
the following result.

\begin{proposition}\label{prop: generic stabilizer of T in max tight support}
 Let $T \in A \otimes B \otimes C$ be a generic tensor with support 
$\calS_{t\text{-}max,m}$. Then $\dim \frakg_T = 1$ and $\frakg_T = \langle L 
\rangle$ where $\langle - \rangle$ denotes the linear span and $L = (U,V,W)$ 
where $U,V,W$ are diagonal with $u^i_i = v^i_i = w^i_i = i - \ell$.
\end{proposition}
\begin{proof}
The Theorem of semicontinuity of dimension of the fiber (see e.g., \cite[Thm. 
1.25]{shaf}) implies that $\dim \frakg_T$ is an upper semicontinuous function. 
In particular, it suffices to prove the statement for a single element $T$ with 
support $\calS_{t\text{-}max,m}$. Suppose that the coefficients of $T$ are 
$T^{ijk} = 1$ for every $(i,j,k) \in \calS_{t\text{-}max,m}$.

We give the proof in the case $m = 2\ell +1$ odd. If $m$ is even, the argument 
is essentially the same, with minor modifications to the index ranges.

Let $d\Phi : \frakgl(A) \oplus \frakgl(B) \oplus \frakgl(C) \to \frakgl(A 
\otimes B \otimes C)$ be the differential of the map $\Phi$ defined at the 
beginning of Section \ref{subsec: tight tensors}. We show that the annihilator of 
$T$ under the action of $ \frakgl(A) \oplus \frakgl(B) \oplus \frakgl(C)$ has 
dimension $3$, and coincides with $\langle L \rangle + \ker (d\Phi)$.

Let $(U,V,W) \in \frakgl(A) \oplus \frakgl(B) \oplus \frakgl(C)$; set $u^i_{i'} 
= 0$ if $i,i' \notin \{ 0 \vvirg 2\ell\}$ and similarly for $v^{j}_{j'}$ and 
$w^k_{k'}$.  Suppose $(U,V,W) \in \frakg_T$, so that every triple $(i,j,k)$ 
provides a (possibly trivial) equation on the entries of $U,V,W$ as follows
\begin{equation}\label{eqn: max support coeff equation}
(i,j,k) \qquad u^{i'}_i + v^{j'}_j + w^{k'}_k = 0
\end{equation}
where $i',j',k'$ are the only integers such that $i' + j + k = i + j' + k = i + 
j+ k' = 3\ell$. Let $\rho = 3\ell - (i+j+k)$; moreover $\rho\in \{ -2\ell \vvirg 
2\ell\}$ and $\rho = i'-i = j'-j = k'-k$. The equations in \eqref{eqn: max 
support coeff equation} can be partitioned into $4\ell+1$ subsets, indexed by 
$\rho = -2\ell \vvirg 2\ell$, so that equations in distinct subsets involve 
disjoint sets of variables. Our goal is to show that the $\rho$-th set of 
equations has no nontrivial solutions if $\rho \neq 0$, whereas the $0$-th set 
of equations has exactly a space of solutions of dimension $3$ which induces 
$(U,V,W) \in \langle L \rangle + \ker( d\Phi)$. Indeed, notice that $(U,V,W) \in 
\langle L \rangle + \ker( d\Phi)$ satisfies all equations in \eqref{eqn: max 
support coeff equation}.

We consider three separate cases: $\rho = 0, 0< \rho < \ell$  and $\rho \geq  
\ell$. The cases $ 0 > \rho > -\ell$ and $\rho \leq  -\ell$ are analogous.

{\bf Case $\rho \geq \ell$.} First, observe that $u^{\rho}_0 = v^{\rho}_0 = 
w^{\rho}_0 =0$. To show this, consider the three equations corresponding to 
$(i,j,k) = (0,0,3\ell-\rho), (0,3\ell-\rho,0)$ and $(3\ell-\rho,0,0)$, which 
give the linear system
\begin{equation}\label{eqn: lin system rho 0}
\left\{\begin{array}{cccccc}
u^{\rho}_0 &+ & v^{\rho}_0 & & & = 0 \\
u^{\rho}_0  &+ &  & & w^{\rho}_0 & = 0 \\
& & v^{\rho}_0  &+ & w^{\rho}_0  & = 0 \\
\end{array} \right.
\end{equation}
in the three unknowns $u^{\rho}_0, v^{\rho}_0 ,  w^{\rho}_0$; this linear system 
has full rank. This shows $u^i_0 = v^j _ 0 = w^k _0 = 0$ if $i,j,k \geq \ell$.

Now fix $q$ with $ \ell > q \geq 1$; we show that $u^{\rho + q}_q = v^{\rho  + 
q}_q = w^{\rho+q}_q=0$. The equation corresponding to $(i,j,k) = (q,0,3\ell 
-\rho  - q )$ is $u^{\rho + q}_q  + v^{\rho}_0  = 0$, which provides $u^{\rho + 
q}_q = 0$ since $v^{\rho}_0 = 0$; similarly $v^{\rho + q}_q = w^{\rho + q}_q = 
0$. If $q = \ell$, then $\rho = \ell$ as well (otherwise $u^{\ell+q}_q$ is 
trivially $0$ because $\ell+q > 2\ell$). In this case, the equations 
corresponding to $(i,j,k) = (0,\ell,\ell),(\ell,0,\ell), (\ell,\ell,0)$ provide 
a linear system similar to \eqref{eqn: lin system rho 0} which provides 
$u^{2\ell}_\ell =
v^{2\ell}_\ell = w^{2\ell}_\ell = 0$.

Apply a similar argument to the case $\rho \leq - \ell$.

{\bf Case $0< \rho < \ell$.} We have $\ell \geq 2$, otherwise this case does not 
occur. First, we show that $u^{2\ell}_{2\ell - \rho} = v^{2\ell}_{2\ell - \rho} 
= w^{2\ell}_{2\ell - \rho} = 0$. This is obtained in two steps. First consider 
the three equations corresponding to the indices $(2\ell - \rho +1, \ell-1,0), 
(2\ell - \rho +2, \ell-2,0),(2\ell - \rho +1, \ell-2,1)$, which are 
\begin{equation}\label{eqn: three eqns in V and W}
\begin{aligned}
& v^{\ell-1+\rho}_{\ell-1} + w^{\rho}_0 = 0 ,\\
& v^{\ell-2+\rho}_{\ell-2} + w^{\rho}_0 = 0 , \\
& v^{\ell-2+\rho}_{\ell-2} + w^{\rho+1}_1 = 0 ;\\
\end{aligned}
\end{equation}
these provide $ v^{\ell-1+\rho}_{\ell-1} + w^{\rho+1}_1 = 0$. Now  the equation 
corresponding to the indices $(2\ell - \rho,\ell-1,1)$, namely 
$u^{2\ell}_{2\ell-\rho} + v^{\ell-1+\rho}_{\ell-1} + w^{\rho+1}_1 = 0$, reduces 
to $u^{2\ell}_{2\ell-\rho} = 0$; similarly, we have $v^{2\ell}_{2\ell-\rho} = 
w^{2\ell}_{2\ell-\rho} = 0$.

This provides the base case for an induction argument. If $q \geq 1$, we show 
that $u^{2\ell - q}_{2\ell - q - \rho} = 0$. This argument is similar to the one 
before: the three equations corresponding to $(2\ell - \rho - (q-1) , \ell+ 
(q-1) ,0), (2\ell - \rho  +1 -(q-1), \ell-1 +(q-1),0),(2\ell - \rho -(q-1) , 
\ell-1 + (q-1),1)$, together with the induction hypothesis, reduce to 
$v^{\ell+(q-1)+\rho}_{\ell+(q-1)} + w^{\rho+1}_1 = 0$. The latter equality, 
together with the equation corresponding to the indices $(2\ell - q - \rho 
,\ell+(q-1),1)$ gives $u^{2\ell-q}_{2\ell-q-\rho} = 0$. Similarly, we have 
$v^{2\ell-q}_{2\ell-q-\rho}$ for every $q = 0 \vvirg
2\ell - \rho$. We conclude that $u^{i+\rho}_{i} = v^{j+\rho}_{j} = 
w^{k+\rho}_{k} = 0$ for every $i,j,k$ and every $0< \rho < \ell$.

Apply a similar argument to the case $0> \rho >- \ell$.

{\bf Case $\rho = 0$.} We may work modulo $\ker (d\Phi) = \langle (\Id_A , - 
\Id_B ,0) , (\Id_A, 0 , -\Id_C)\rangle$. In particular, we may assume $V,W$ 
satisfy $\trace(V) = \trace(W) = 0$. Consider all equations $(\ell, j ,k )$ so 
that $j + k = 2\ell$. Adding them up and using the traceless condition, we have 
$u^\ell _\ell = 0$ and therefore $v^{\ell+q}_{\ell+q} = - w^{\ell-q}_{\ell - q}$ 
for $q = - \ell \vvirg \ell$. Let $\xi = u^{\ell +1}_{\ell +1}$. Then for every 
$q$, the equation $(\ell+1, \ell + q-1, \ell -q)$ gives $v^{\ell + 
q-1}_{\ell+q-1} + \xi = - w ^{\ell-q}_{\ell-q} = v^{\ell+q}_{\ell+q}$, so that 
one has $v^{\ell+q}_{\ell+q} = v^\ell_\ell + q\xi$ and similarly 
$w^{\ell+q}_{\ell+q} = w^\ell_\ell + q\xi$. Since $V$ and $W$ are traceless, we 
obtain $v^\ell_\ell = w^\ell_\ell = 0$ and $v^{\ell+q}_{\ell+q} = 
w^{\ell+q}_{\ell+q} = q\xi$. In particular, by adding up the equations for the 
form $(\ell+q,\ell-q,\ell)$ for $q = -\ell \vvirg \ell$, we observe that $U$ is 
traceless as well, and by a 
similar argument $u^{\ell+q}_{\ell+q} = q\xi$ as well, so that $(U,V,W) = L$. 
This shows that modulo $\ker (d \Phi)$ we have $\frakg_T = \langle L \rangle$, 
and this concludes the proof.
\end{proof}

\subsection{Oblique tensors}

Recall that a tensor $T$ is oblique if there are bases such that $\supp(T)$ is 
an antichain in $[\bfa]\times [\bfb]\times [\bfc]$  under the partial ordering 
induced by three total orders on $[\bfa], [\bfb], [\bfc]$. The original 
definition of oblique considers the three sets $[\bfa], [\bfb], [\bfc]$ with the 
natural ordering induced by $\bbN$. Our definition allows reordering in the 
index ranges of each factor: this does not affect the resulting class of 
tensors, and provides the following useful fact.

\begin{remark}\label{rmk: tight is oblique}
 Every tight set is oblique. Let $\calS \subseteq [\bfa]\times [\bfb]\times 
[\bfc]$ be a tight set. After permuting the elements of $[\bfa],[\bfb],[\bfc]$, 
we may assume that $\tau_A,\tau_B,\tau_C$ are strictly increasing. Suppose 
$\calS$ is not an antichain in $[\bfa]\times [\bfb]\times [\bfc]$ and let 
$(i_1,j_1,k_1), (i_2,j_2,k_2) \in \calS$ distinct such that $i_1 \leq i_2, j_1 
\leq j_2, k_1 \leq k_2$, with at least one strict inequality. Therefore 
$\tau_A(i_1) + \tau_B(j_1) + \tau_C(k_1) < \tau_A(i_2) + \tau_B(j_2) + 
\tau_C(k_2)$, in contradiction with the assumption that $\calS$ is tight.
\end{remark}

In order to give some insights on oblique subsets, we introduce terminology from 
\cite{Proc}. To avoid confusion with tensor rank, we  use \lq\lq poset rank\rq\rq\ where
Proctor uses \lq\lq rank\rq\rq .

\begin{definition}[{\bf \cite{Proc}}]
Let $(P, \prec)$ be a poset and let $x,y \in P$. The element $x$ {\it covers} 
$y$ if $y \prec x$ and there does not exist $z\in P$ such that $y \prec z \prec 
x$. A {\it ranked poset} $P$ of {\it length $r$} is a poset $P$ with a partition 
$P = \bigsqcup_{i=0}^{r} P_i$ into $r+1$ {\it poset ranks} $P_i$, such that elements 
in $P_i$ cover only elements in $P_{i-1}$. (Note that the elements in $P_i$ do not have to cover all the elements in $P_{i-1}$.) A ranked poset of length $r$ is {\it 
poset rank symmetric} if $|P_i| = |P_{r-i}|$ for $1\leq i < r/2$. It is {\it poset rank 
unimodal} if $|P_1| \leq |P_2| \leq \cdots \leq |P_{h_0}| $ and $|P_{h_0}| \geq 
|P_{h_0+1}| \geq \cdots \geq |P_{r+1}|$, for some $1\leq h_0 \leq r+1$.

 A poset is {\it Peck} if it is poset rank symmetric, poset rank unimodal and for every 
$\ell\geq 1$ no union of $\ell$ antichains contains more elements than the union 
of the $\ell$ largest poset ranks of $P$.
\end{definition}

\begin{example}
For every ${\bf a}$, the poset $[{\bf a}]$ is ranked of length ${\bf a}-1$ and 
it is Peck.
\end{example}

Using representation-theoretic methods, Proctor \cite[Thm.  2]{Proc} showed that 
products of Peck posets are Peck posets, with respect to the natural 
\emph{product ordering} and with poset rank function defined by the sum of the poset rank 
functions of the factors; in particular $[{\bf a}]\times [{\bf b}]\times [{\bf 
c}]$ is Peck according to the induced partial ordering on the product and the 
poset rank function is given by $h(i,j,k) = i+j+k$.
 
\begin{remark}\label{rmk: oblique when tight}
Oblique supports entirely contained in a single poset rank are tight. More explicitly, 
let $P = [{\bf a}]\times [{\bf b}]\times [{\bf c}]$. Every oblique tensor $T$ 
whose support $\cS_{T}$ is an antichain in some poset rank $P_h$ of $P$ is tight. In 
particular $\calS_{t\text{-}max,m}$ coincides with $P_{3\ell}$, with $m = 
2\ell+1$ or $m = 2\ell$; using Proctor's terminology, this corresponds to the 
$\fraksl_2$-weight space of weight $0$ in the representation $\bbC^P = \bbC^\bfa 
\otimes \bbC^\bfb \otimes \bbC^{\bfc}$ where the factors are regarded as 
irreducible $\fraksl_2$-representations. 
\end{remark}

The following is a slightly stronger version of Theorem \ref{thm: maximum 
support}(ii): 

\begin{theorem}\label{maxsuppoblique}
 Let $\aaa\leq \bbb\leq \ccc$ and let $\cS\subset [\aaa]\times[\bbb]\times 
[\ccc]$ be oblique. Then
\[
\vert \calS \vert \leq \left\{
\begin{array}{ll} 
{\bfa}{\bfb} - \lfloor \frac{({\bfa}+{\bfb}-{\bfc})^2}{4} \rfloor & \text{if 
$\bfa + \bfb \geq \bfc$} \\ 
\bfa\bfb &  \text{if $\bfa + \bfb \leq \bfc$}.
\end{array} \right.
\]
Moreover, in all cases there exist $\calS$ such that equality holds.
\end{theorem}
\begin{proof}
 Since $P = [\bfa] \times [\bfb] \times [\bfc]$ is Peck, the cardinality of a 
maximal antichain is upper bounded by the maximal poset rank subset: since a Peck set 
is unimodular, the maximal poset rank is the central one, namely $P_{h_{max}} = \{ 
(i,j,k) : i + j + k = h_{max} \}$ where $h _{max} = \lfloor \frac{\bfa + \bfb + 
\bfc - 3}{2} \rfloor$ (and equivalently $\lceil \frac{\bfa + \bfb + \bfc - 3}{2} 
\rceil$).

If $\bfa + \bfb < \bfc$ then for every $(i,j) \in [\bfa] \times [\bfb]$ there 
exists $k \in [\bfc]$ such that $i+j+k = h_{max}$, so $\vert P _ {h_{max}} \vert 
= \bfa \bfb$ and the statement of the theorem holds. 

Now suppose $\bfa + \bfb \geq \bfc$. Let $\psi: P \rightarrow [{\bfa}] \times 
[{\bf b}]$ be the projection onto the first two factors. Note that $\psi$ 
restricted to each $P_h$ is injective because $P_h$ is an antichain. Then $\vert 
P_{h_{max}} \vert  = \vert \psi(P_{h_{max}})\vert$. We compute the number of 
elements of $\psi(P_{h_{max}})$. Consider its complement in $[\bfa ] \times 
[\bfb]$, that is the set of pairs $(i,j) \in [{\bfa}]\times [{\bfb}]$ for which 
there is no $k\in [\bfc]$ with $i+j+k = h_{{max}}$. Since $0 \leq k \leq 
\bfc-1$, these are exactly pairs $(i,j)$ satisfying one of the following 
conditions:
\begin{enumerate}
\item[(i)] $i + j \leq h_{max} - (\bfc - 1) - 1$, that is $i+j \leq \lfloor 
\frac{\bfa + \bfb - \bfc - 3}{2} \rfloor$;
\item[(ii)] $ h_{max} \leq i+j-1$, that is $\lfloor \frac{\bfa + \bfb+ 
\bfc-1}{2} \rfloor \leq i+j $.
\end{enumerate}

Notice that (i) and (ii) are mutually exclusive. Let $\theta = \lfloor 
\frac{\bfa + \bfb - \bfc - 3}{2} \rfloor$. For every $i = 0 \vvirg \theta$, and 
every $j = 0 \vvirg \theta - i$, we have $i+j \leq \theta$; this gives $1 + 2 + 
\cdots + (\theta+1) = \binom{\theta + 2}{2}$ pairs $(i,j)$ satisfying condition 
(i). Now, let $i ' = \bfa - 1 - i$ and $j' = \bfb - 1 - j$: condition (ii) can 
be rephrased as $\bfa + \bfb - 2 - \lfloor \frac{\bfa + \bfb+ \bfc-1}{2} \rfloor 
\geq i' + j'$ which in turn becomes $i' + j' \leq \eta$ where $\eta = \lceil 
\frac{\bfa + \bfb - \bfc - 3}{2} \rceil$; this provides $\binom{\eta+2}{2}$ 
pairs $(i',j')$ which correspond to $\binom{\eta+2}{2}$ pairs $(i,j)$ satisfying 
(ii). We conclude that the complement of $\psi(P_{h_{max}})$ in $[\bfa] \times 
[\bfb]$ consists of $\binom{\theta+2}{2} + \binom{\eta+2}{2}$ elements. To 
conclude, observe $\binom{\theta+2}{2} + \binom{\eta+2}{2} = \lfloor 
\frac{({\bfa}+{\bfb}-{\bfc})^2}{4} \rfloor$.
\end{proof}

\begin{remark} The above proof is modeled on the proof of  \cite[Thm. 
6.6]{MR882307}.
\end{remark}

Choosing $\bfa = \bfb = \bfc = m$ in Theorem \ref{maxsuppoblique}, one obtains 
the bound of Theorem \ref{thm: maximum support}(ii). Since every tight tensor is 
oblique, the same bound holds for tight tensors. Since $\calS_{t\text{-}max,m}$ 
from Example \ref{example: max tight support} is a tight support of cardinality 
$\lceil \frac{3}{4}m^2 \rceil$ (which in fact corresponds to a maximal antichain 
as observed in Remark \ref{rmk: oblique when tight}), we obtain that the bound 
is sharp both in the oblique and in the tight case.

\subsection{Free tensors}

We recall that a subset $\calS \subseteq [\bfa] \times [\bfb]\times [\bfc]$ is 
free if any two triples $(i,j,k)$, $(i',j',k')$ in $\calS$ differ on at least 
two entries.  

\begin{remark}\label{rmk: oblique is free}
Every oblique support is free. Let $\calS$ be an oblique support and suppose it 
is not free. Without loss of generality, $\calS$ contains two triples of the 
form $s_1 = (i,j,k_1)$ and $s_2 = (i,j,k_2)$ for some $k_1,k_2$. But then, if 
$k_1 \leq k_2$ then $s_1 \leq s_2$ and if $k_2 \leq k_1$ then $s_2 \leq s_1$, 
therefore $\calS$ is not an antichain, providing a contradiction.
\end{remark}

\begin{example}[A free support of cardinality $m^2$]\label{example: maximal free 
support}
We obtain a free support of cardinality $m^2$ by completing the support 
$\calS_{t\text{-}max,m}$ in a circulant way. More precisely, let $m \geq 0$ be 
odd with $m = 2\ell+1$. Define 
\[
\calS_{f\text{-}max,m} = \{(i,j,k) : i+j +k \equiv \ell \mod m \} \subseteq 
[m]\times [m]\times [m].
\]
Notice that in the range where $ \ell \leq j+k <3\ell$, then $i = 0 \vvirg 
2\ell$ with $i+j+k = 3\ell$, recovering the structure of 
$\calS_{t\text{-}max,m}$.
\end{example}

It is immediate from the definition that the cardinality of a free support is at 
most $m^2$: indeed, any $m^2 + 1$ elements would have at least two triples 
$(i,j,k)$ with the same $(i,j)$. This observation, together with Example 
\ref{example: maximal free support}, completes the proof of Theorem \ref{thm: 
maximum support}(iii).

\subsection{Proof of Theorem \ref{thm: dimension of classes of tensors}}

We first describe the general construction that will be used in the proof.

Fix a vector space $V$ and let $1\leq \kappa \leq \dim V-1$. Let $G(\kappa,V)$ 
denote the Grassmannian of $\kappa$-planes through the origin in $V$ and let 
$\pi_G: \calT \to G(\kappa,V)$ denote the tautological subspace bundle of 
$G(\kappa,V)$, i.e., the vector bundle whose fiber over a $\kappa$-dimensional 
plane $E \in G(\kappa,V)$ is $E$ itself. Let $\pi_V: \calT \to V$ denote the 
projection to $V$, that is $\pi_V : (E, v) \mapsto v$ for every $E \in 
G(\kappa,V)$ and $v \in E \subseteq V$.

If $Z\subset G(\kappa,V)$ is a subvariety of dimension $z$, then $\dim 
\pi_G^{-1}(Z)=z+\kappa$.  Consequently,  $\dim \pi_V(\pi_G^{-1}(Z))\leq 
z+\kappa$. The action of a group on $V$ naturally induces an action on 
$G(\kappa,V)$ and the vector bundle $\calT$ can be restricted to orbits and 
orbit-closures of such an action.

We will use this construction in the setting where $V = A \otimes B \otimes C$, 
and $Z$ is the $GL(A) \times GL(B) \times GL(C)$-orbit closure of the linear 
space consisting of all tensors with a given support; we refer to such linear 
space as the \emph{span of a support}.

The variety $\bar{\Tight}_m$, (resp. $\bar{\Oblique}_m$, $\bar{\Free}_m$) is a 
union of subvarieties of the form $\pi_V(\pi_G^{-1}(Z))$ with 
$Z=\bar{GL(A)\times GL(B)\times GL(C)\cdot E}$ and  $E$ is the span of a tight 
(resp. oblique, free) support in some given bases, regarded as an element of 
$G(\dim E, A \otimes B \otimes C)$. In particular, we have the following

\begin{lemma}\label{lemma: dimension Tight via fiber}
\[
 \dim \bar{\Tight}_m = \max \left\{ \dim \pi_V(\pi_G^{-1}(Z)) : 
\begin{array}{ll} Z = \bar{GL(A)\times GL(B)\times GL(C)\cdot E} \\ \text{ for 
some $E \in G(\kappa, A \ot B \ot C)$ span of a tight support} \end{array} 
\right\},  
\]
and similarly for $\bar{\Oblique}_m$ and $\bar{\Free}_m$.
\end{lemma}
\begin{proof}
Every tight tensor is in the $GL(A) \times GL(B) \times GL(C)$ orbit of a tight 
tensor in a fixed basis. Moreover, the number of tight supports in a fixed basis 
is finite. This implies that the irreducible components of the variety 
$\bar{\Tight}_m$ have the form $\pi_V(\pi_G^{-1}(Z))$, where $ Z= 
\bar{GL(A)\times GL(B)\times GL(C)\cdot E}$ for some linear space $E$ which is 
the span of a non-extendable tight support. 

Since the number of supports is finite, $\dim \bar{\Tight}_m$ is just the 
dimension of the largest orbits.

The same holds for $\bar{\Oblique}_m$ and $\bar{\Free}_m$.
\end{proof}

The following lemma gives the dimension of the orbit closure of the span of a 
concise free support $E$. Since from Remark \ref{rmk: tight is oblique} every 
tight support is oblique (up to reordering the bases) and from Remark \ref{rmk: 
oblique is free} every oblique support is free, the same result applies to tight 
and oblique supports. 
\begin{lemma} \label{lemma: annihilator of free concise support}
Let $E\in G(\kappa,A\ot B\ot C)$ be the span of a concise free support and let 
$Z = \bar{GL(A)\times GL(B)\times GL(C)\cdot E} \subseteq G(\kappa , A \otimes B 
\otimes C)$. Then $\dim Z = \aaa^2+\bbb^2+\ccc^2- (\aaa+\bbb+\ccc)$.
\end{lemma}
\begin{proof}
We show that the affine tangent space to $Z$ at $E$ in the Plucker embedding of 
$G(\kappa,A \otimes B \otimes C)$ in $\bbP \Lambda^\kappa (A \otimes B \otimes 
C)$ has dimension exactly $\aaa^2+\bbb^2+\ccc^2- (\aaa+\bbb+\ccc) + 1$; in the 
following, let $\hat{G}(\kappa,A\otimes B \otimes C) \subseteq 
\Lambda^\kappa(A\otimes B \otimes C)$ be the cone over $G(\kappa,A\otimes B 
\otimes C)$. The affine tangent space to $Z$ at $E$ is $\hat{T} _E Z = \left\{ 
\left( \frakgl(A) \oplus \frakgl(B) \oplus \frakgl(C) \right). E \right\}$, 
which is naturally a subspace of $\Lambda^\kappa (A \otimes B \otimes C)$. Here 
$E$ is identified with the element $\bigwedge_{s = 1}^\kappa (a_{i_s} \otimes 
b_{j_s} \otimes c_{k_s}) \in \hat{G}(\kappa,A \otimes B \otimes C)$, where 
$\{(i_s , j_s, k_s) : s = 1 \vvirg \kappa\}$ is the free support defining $E$. 

Let $(X,Y,Z) \in \frakgl(A) \oplus \frakgl(B) \oplus \frakgl(C)$ for three $m 
\times m$ matrices $X,Y,Z$. If $X,Y,Z$ are diagonal, then $(X,Y,Z) . E = E$ up 
to scale. Thus $\dim \hat{T} _E Z \leq \aaa^2+\bbb^2+\ccc^2- (\aaa+\bbb+\ccc) + 
1$. In order to show equality, it suffices to observe that the vectors of the 
form $(X,Y,Z) . E$, where $X,Y,Z$ are three matrices which are all $0$ except in 
a single off-diagonal entry in one of them, are linearly independent and span a 
subspace of $\Lambda^{ {\kappa}} (A \otimes B \otimes C)$ which does not contain 
$E$; in particular, such a subspace has dimension $\aaa^2+\bbb^2+\ccc^2- 
(\aaa+\bbb+\ccc) $.

For every $L = (X,Y,Z)$ having exactly one off-diagonal nonzero entry, we 
observe that $L.E \neq 0$ and that every summand in the expansion of $L. E$ as 
sum of basis vectors of $\Lambda^\kappa (A \otimes B \otimes C)$ differs from 
$E$ in exactly one factor: $L.E \neq 0$ follows immediately by freeness, while 
the second condition is realized whenever $E$ is spanned by basis vectors. In 
particular, the subspace of $\Lambda^\kappa(A \otimes B \otimes C)$  generated 
by the $L.E$'s does not contain $E$.

The same argument shows that the $L.E$'s are linearly independent. Indeed, 
suppose $L_1, L_2$ both have exactly one nonzero entry and suppose that $L_1.E$ 
and $L_2.E$ both have a summand $\Theta$ in their expansion as sum of basis 
vectors of $\Lambda^\kappa(A \otimes B \otimes C)$. Regard $\Theta$ as an 
element of $\hat{G}(\kappa,A \otimes B \otimes C)$ (it is the wedge product of a 
set of basis vectors), namely a coordinate $\kappa$-plane in $A \otimes B 
\otimes C$. There are exactly two basis elements $v = a_{i_0} \otimes b_{j_0} 
\otimes c_{k_0}$, $v'= a_{i_0'} \otimes b_{j_0'} \otimes c_{k_0'}$ such that $v 
\in E \setminus \Theta$ and $v' \in \Theta \setminus E$ and two of the three 
factors of $v$ coincide with the corresponding factors of $v'$. There is a 
unique element of $L \in \frakgl(A) + \frakgl(B) + \frakgl(C) $ having exactly 
one off-diagonal entry such that $L.v = v'$, which guarantees $L = L_1 = L_2$. 
In particular, all the $L.E$'s are linearly independent and this concludes the 
proof.
\end{proof}

In particular, from Lemma \ref{lemma: annihilator of free concise support}, one 
immediately obtains $\dim \pi_G^{-1}(Z)$ when $Z$ is the orbit-closure of the 
span of a concise free support $\calS$. If $\bfa=\bfb=\bfc = m$, we have 
\begin{equation}\label{eqn: dim preimage over grassmannian}
\dim \pi_G^{-1}(Z) = 3m^2 - 3m + \vert \calS \vert.
\end{equation}

Equation \eqref{eqn: dim preimage over grassmannian} guarantees that to prove 
Theorem \ref{thm: dimension of classes of tensors}, it suffices to determine a 
tight (resp. oblique, free) support $\calS$ such that $\vert \calS \vert = 
\lceil \frac{3}{4} m^2 \rceil$ (resp. $\lceil \frac{3}{4} m^2 \rceil$, $m^2$) 
and with the property that the projection $\pi_V : \calT\vert_Z \to A \otimes B 
\otimes C$ is generically finite-to-one. Indeed, if the projection $\pi_V$ is 
finite-to-one on $\calT\vert_Z = \pi_G^{-1}(Z)$, we have $\dim \pi_V ( 
\pi_G^{-1}(Z) ) = \dim \pi_G^{-1}(Z) = 3m^2 - 3m + \vert \calS \vert$ and 
considering $\calS$ with $\vert \calS \vert = \lceil \frac{3}{4} m^2 \rceil$ in 
the tight case, $\vert \calS \vert = \lceil \frac{3}{4} m^2 \rceil$ in the 
oblique case, and $\vert \calS \vert = m^2$ in the free case, via Lemma 
\ref{lemma: dimension Tight via fiber} we obtain the dimensions indicated in 
Theorem \ref{thm: dimension of classes of tensors}.

For the tight and oblique cases, we consider $\calS = \calS_{t\text{-}max,m}$ 
from Example \ref{example: max tight support}, and for the free case we consider 
$\calS = \calS_{f\text{-}max,m}$ from Example \ref{example: maximal free 
support}.

{\bf Tight and Oblique case.} Let $Z = \bar{GL(A) \times GL(B) \times GL(C) 
\cdot E} \subseteq G(  \lceil \frac{3}{4} m^2 \rceil, A \otimes B \otimes C )$ 
where $E = \langle \calS_{t\text{-}max,m} \rangle$ is the linear space of 
tensors supported at $\calS_{t\text{-}max,m}$. We prove that the fiber of 
$\pi_V$ at a generic point of $E$ is $0$-dimensional. From Proposition 
\ref{prop: generic stabilizer of T in max tight support}, we have $\dim G_T = 1$ 
and in particular the connected component of the identity in $G_T$ is a 
$1$-parameter subgroup which is diagonal in the fixed basis; let $\Gamma_E$ be 
this subgroup.

The fiber of $\pi_V|_{\pi^{-1}_G(Z)}$ over a tensor $T$ is the subset of 
$\calT\vert_Z$ defined by $Y_T = \{ (F,T) : F \in Z, T \in F\}$.  Our goal is to 
show that if $T$ is generic, then $Y_T$ is finite. If $(F,T) \in Y_T$, with $F 
\neq E$, then $F = gE$ for some $g= (g_A,g_B,g_C) \in GL(A) \times GL(B) \times 
GL(C)$. At least one of $g_A,g_B,g_C$ is not diagonal in the chosen basis, 
otherwise $gE = E$. The linear space $F$ is a tight support in the bases 
$g_A(a_i), g_B(b_j),g_C(c_k)$; in particular the one-parameter subgroup 
$\Gamma_F = g^{-1} \Gamma_E g $ stabilizes every tensor in $F$ and in particular 
$T$. We deduce $\Gamma_F \subseteq G_T$. Notice that $\Gamma_F \neq \Gamma_E$, 
because $\Gamma_F$ is not diagonal in the bases $a_i,b_j,c_k$. Now, $\Gamma_E$ 
and $\Gamma_F$ are two distinct $1$-parameter subgroups of $G_T$, which implies 
$\dim G_T \geq 2$, in contradiction with Proposition \ref{prop: generic 
stabilizer of T in max tight support}. This shows that $\pi_V|_{\pi_G^{-1}(Z)}$ 
is generically 
finite-to-one.

{\bf Free case.} Let $Z = \bar{GL(A) \times GL(B) \times GL(C) \cdot E} 
\subseteq G( m^2 ,A \otimes B \otimes C )$ where $E = \langle 
\calS_{f\text{-}max,m} \rangle$ is the linear space of tensors supported at 
$\calS_{f\text{-}max,m}$. Let $T$ be a tensor in $E$ such that $\supp(T) 
\subseteq \calS_{t\text{-}max,m} \subseteq \calS_{f\text{-}max,m}$. The tensor 
$T$ is tight and the same argument that we followed in the previous case shows 
that the fiber of $\pi_V$ is finite at $T$. By semicontinuity of dimension of 
the fibers (see e.g., \cite[Thm. 1.25]{shaf}), $\pi_V$ has $0$-dimensional fiber 
at the generic point of $E$ and therefore $\pi_V|_{\pi_G^{-1}(Z)}$ is 
generically finite-to-one.

Via equation \eqref{eqn: dim preimage over grassmannian}, we now conclude the 
proof of Theorem \ref{thm: dimension of classes of tensors}:
\[
\begin{array}{rll}
 \dim \bar{\Tight}_m &=  \dim \pi_V ( \pi^{-1}_G (Z) ) = \dim \pi^{-1}_G (Z)  & 
= 3m^2 - 3m + \vert \calS_{t\text{-}max,m} \vert  = \\
 & &= 3m^2 - 3m + \lceil \frac{3}{4} m^2 \rceil, \\
 \dim \bar{\Oblique}_m &=  \dim \pi_V ( \pi^{-1}_G (Z) ) = \dim \pi^{-1}_G (Z)  
&= 3m^2 - 3m + \vert \calS_{t\text{-}max,m} \vert =  \\ & & =3m^2 - 3m + \lceil 
\frac{3}{4} m^2 \rceil, \\
 \dim \bar{\Free}_m &=  \dim \pi_V ( \pi^{-1}_G (Z) ) = \dim \pi^{-1}_G (Z)  &= 
3m^2 - 3m + \vert \calS_{f\text{-}max,m} \vert =  \\ & & = 3m^2 - 3m + m^2. 
\end{array}
\]

\subsection{Tight, oblique and free in small dimension and inclusions among 
classes of tensors}
We saw that every tight tensor is oblique and every oblique tensor is free. 
The inclusions  $\bar{\Oblique}_m \subseteq \bar{\Free}_m$ are strict 
since the two varieties have different dimensions. The varieties $\bar{\Tight}_m$ 
and $\bar{\Oblique}_m$ have the same dimension. 

In this subsection  we show that $\bar{\Tight}_3 = \bar{\Oblique}_3$, and that  the inclusion 
$\bar{\Tight}_m \subseteq \bar{\Oblique}_m$ is strict for $m \geq 4$.

\begin{proof}[Proof of Theorem \ref{m3thm}]  The dimensions follow immediately from
Theorem \ref{thm: dimension of classes of tensors}, so it remains to prove $\bar{\Tight}_3 =\bar{\Oblique}_3$.
 This statement is proved via a computer calculation. There are $144$ maximal 
antichains in $[3] \times [3] \times [3]$; only $80$ of these are concise, in 
the sense that generic tensors with the corresponding support are concise. The 
group $\frakS_3 \times \bbZ_2$ acts on $[3] \times [3] \times [3]$, where 
$\frakS_3$ permutes the factors and $\bbZ_2$ maps $(i,j,k)$ to $(2-i,2-j,2-k)$. 
The induced action on subsets of $[3] \times [3]\times [3]$ preserves tight 
supports and antichains. In particular, without loss of generality, it suffices 
to prove the statement for an antichain in each orbit of $\frakS_3 \times 
\bbZ_2$. There are $13$ such orbits. The following are representatives for the 
orbits:
 \begin{align*}
\calS_{1} &= \{{(0,0,2), (0,1,1), (1,0,1), (2,2,0)}\}, \\
\calS_{2} &= \{{(0,0,2), (0,2,0), (1,1,1), (2,0,0)}\}, \\
\calS_{3} &= \{{(0,0,2), (0,2,1), (1,1,0), (2,0,1)}\}, \\
\calS_{4} &= \{{(0,0,2), (0,2,1), (1,2,0), (2,1,1)}\}, \\
\calS_{5} &= \{{(0,0,2), (0,1,1), (0,2,0), (1,0,1), (2,1,0)}\}, \\
\calS_{6} &= \{{(0,0,2), (0,1,1), (1,0,1), (1,2,0), (2,1,0)}\}, \\
\calS_{7} &= \{{(0,0,2), (0,1,1), (1,2,0), (2,0,1), (2,1,0)}\}, \\
\calS_{8} &= \{{(0,0,2), (0,2,0), (1,1,1), (2,0,1), (2,1,0)}\}, \\
\calS_{9} &= \{{(0,0,2), (0,2,1), (1,1,1), (2,0,1), (2,2,0)}\}, \\
\calS_{10} &= \{{(0,1,1), (1,0,2), (1,2,0), (2,0,1), (2,1,0)}\}, \\
\calS_{11} &= \{{(0,0,2), (0,1,1), (0,2,0), (1,0,1), (1,1,0), (2,0,0)}\}, \\
\calS_{12} &= \{{(0,0,2), (0,2,1), (1,1,1), (1,2,0), (2,0,1), (2,1,0)}\}, \\
\calS_{13} &= \{{(0,1,2), (0,2,1), (1,0,2), (1,1,1), (1,2,0), (2,0,1), 
(2,1,0)}\}. \\
\end{align*}
For each of these, we provide the functions $\tau_A,\tau_B,\tau_C$ which 
guarantee tightness. We record the functions in the following table
\begin{equation*}
 \begin{array}{l|c|c|c}
&  ( \tau_A(0), \tau_A(1), \tau_A(2) ) & ( \tau_B(0), \tau_B(1), \tau_B(2) ) & ( 
\tau_C(0), \tau_C(1), \tau_C(2) )  \\ \midrule
\calS_1 & (-2,-3,1)& (2,1,0) & (-1,1,0) \\
\calS_2 & (1,-2,2) & (-1,1,0) & (-1,1,0) \\
\calS_3 & (-1,2,-2) & (1,2,0) & (-4,1,0) \\
\calS_4 & (-2,2,-1) & (2,-1,0) & (-2,2,0) \\
\calS_5 & (-1,-3,2) & (1,-1,2) & (-1,2,0) \\
\calS_6 & (0,-1,2) & (0,-1,2) & (-1,1,0)  \\
\calS_7 & (2,-3,1) & (2,-3,1) & (2,1,0) \\
\calS_8 & (2,-2,1) & (-2,1,0) & (-2,1,0) \\
\calS_9 & (0,1,2) & (0,1,2) & (-4,-2,0) \\
\calS_{10} & (-2,2,1) & (-2,1,0) & (-2,1,0) \\
\calS_{11} & (-2,1,4) & (-2,1,4) & (-2,1,4) \\
\calS_{12} & (-2,1,4)& (-2,1,4)& (-5,-2,4 )\\
\calS_{13} & (-1,0,1) & (-1,0,1) & (-1,0,1) \\
 \end{array}
\end{equation*}
This shows that every oblique support in $[3] \times [3] \times [3]$ is tight; 
in particular, every oblique tensor is tight and $\bar{\Tight} _3 = 
\bar{\Oblique}_3$.
\end{proof}

\begin{proposition}\label{prop: tight strict in oblique}
Let $T \in A \otimes B \otimes C$ with $\bfa = \bfb = \bfc = 4$ be the tensor
\begin{align*}
T = & a_0\otimes b_2\otimes c_3+a_0\otimes b_3\otimes c_2+a_1\otimes
b_0\otimes c_3+a_1\otimes b_1\otimes c_2+a_1\otimes b_2\otimes c_1 \\
+&a_1\otimes b_3\otimes c_0+a_2\otimes b_1\otimes c_1+a_2\otimes b_2\otimes
c_0+a_3\otimes b_0\otimes c_2+a_3\otimes b_1\otimes c_0.
\end{align*}
Then $T$ is oblique and not tight.
\end{proposition}
\begin{proof}
The proof of obliqueness is directly by observing that the support
\[
\calS = \{ (0,2,3), (0,3,2) , 
(1,0,3),(1,1,2),(1,2,1),(1,3,0),(2,1,1),(2,2,0),(3,0,2),(3,1,0)\} 
\]
is an antichain in $[4]\times [4]\times [4]$.

On the other hand $T$ is not tight: a direct calculation shows that its 
annihilator $\frakg_T$ is trivial.
\end{proof}

Relying on the additivity result of Theorem \ref{thm: propagation of 
symmetries}(i), one obtains that the inclusion $\bar{\Tight}_m \subseteq 
\bar{\Oblique}_m$ is strict for every $m \geq 4$. To see this, let $T_4$ be the 
tensor of Proposition \ref{prop: tight strict in oblique} and define $T_m = T_4 
\oplus \Mamu{1}^{\oplus m-4}$. Then $T_m$ is oblique but it is not tight.

We conclude this section with a result on border rank of tight tensors. Let 
$Seg: \bbP A \times \bbP B \times \bbP C \to \bbP (A\otimes B\otimes C)$, 
$Seg([u],[v],[w]) = [u \otimes v \otimes w]$, be the \emph{Segre embedding}, 
whose image $Seg(\bbP A \times \bbP B \times \bbP C )$ is the variety of rank 
one tensors. Let $\sigma_r(Seg(\bbP A \times \bbP B \times \bbP C)) \subseteq 
\bbP (A\otimes B \otimes C)$ be the $r$-th \emph{secant variety} of $Seg(\bbP A 
\times \bbP B \times \bbP C)$, that is the variety of tensors of border rank at 
most $r$.

\begin{proposition}\label{prop: secants and tight}
 Let $\bfa = \bfb = \bfc = m$. Then
\begin{itemize}
 \dotitem $\sigma_m(Seg(\bbP A \times \bbP B \times \bbP C)) \subseteq 
\bar{\Tight}_m$. In other words, $\ur(T)\leq m$ implies $T\in \bar{\Tight}_m$.
 \dotitem $\sigma_{m+1}(Seg(\bbP A \times \bbP B \times \bbP C)) \not\subseteq 
\bar{\Tight}_m$.
  In other words, a general tensor $T$ with $\ur(T)\geq m+1$ is not tight.
\end{itemize}
\end{proposition}
\begin{proof}
 If $r \leq m$, then $\sigma_r(Seg(\bbP A \times \bbP B \times \bbP C)) = 
\bar{(GL(A) \times GL(B) \times GL(C)) \cdot \Mamu{1}^{\oplus r}}$, where $ 
\Mamu{1}^{\oplus r} = \sum_{0}^{r-1} a_i \otimes b_i \otimes c_i$. Since 
$\Mamu{1}^{\oplus r}$ is tight, we have $\sigma_m(Seg(\bbP A \times \bbP B 
\times \bbP C)) \subseteq \bar{\Tight}_m$.
 
 Let $T_{std,m} = \Mamu{1}^{\oplus m} + \left(\sum_1^m a_i \right) \otimes 
\left(\sum_1^m b_i \right) \otimes \left(\sum_1^m c_i \right)$. From the 
expression one sees that $T_{std,m} \in\sigma_{m+1}(Seg(\bbP A \times \bbP B 
\times \bbP C))$. A direct calculation shows that the annihilator 
$\frakg_{T_{std,m}}$ is trivial, therefore $T_{std,m}$ is not tight. We conclude 
$\sigma_{m+1}(Seg(\bbP A \times \bbP B \times \bbP C)) \not\subseteq 
\bar{\Tight}_m$.
\end{proof}

\section{Compressibility of Tight tensors}\label{section: compressibility}
In this section we briefly review Strassen's spectral theory and his support functionals  in
order to state the original version of Conjecture \ref{strconj} and to relate it to the notion
of {\it compressibility} of tensors.

\subsection{Strassen's spectral theory}
In \cite{Str:AsySpectrumTensorsExpMatMult,MR882307,Str:AsySpectrumTensors,MR1089800}, Strassen   proved that asymptotic degeneration of tensors (and in particular the asymptotic rank) is   captured by  what he named the \emph{asymptotic spectrum of tensors}. What follows is a brief description of the theory, see  \cite{2017arXiv170907851C,LCBMS} for   extensive discussions. Let $\calT = \varinjlim_m \bbC^m \otimes \bbC^m \otimes \bbC^m$ be the direct limit defined by fixed inclusions of $\bbC^m \subseteq \bbC^{m+1}$. The set $\calT$ is a semi-ring  under the operations of direct sum and Kronecker product. There is a natural preorder on $\calT$  given by \emph{asymptotic degeneration}: $T' \lesssim T$ if there exists a sequence $\{\alpha_N \} \in o(N)$ such that, ${T'}^{\boxtimes N}$ is a degeneration of $T^{\boxtimes N + \alpha_N}$ for every $N$. Strassen proved that asymptotic degeneration is controlled by \emph{spectral points}: these are real-valued semi-ring homomorphisms which are monotone under degeneration. In symbols, a spectral point is a function $\phi : \calT \to \bbR_+$, such that $\phi(\Mone)=1$, $\phi(T_1 \oplus T_2) = \phi(T_1) + \phi(T_2)$, $\phi(T_1 \boxtimes T_2) = \phi(T_1)\phi(T_2)$ and $\phi(T_1) \leq \phi(T_2)$ whenever $T_1$ is a degeneration of $T_2$. Strassen proved that $T_1 \lesssim T_2$ if and only if $\phi(T_1) \leq \phi(T_2)$ for all spectral points $\phi$. 

Since spectral points are semi-ring homomorphisms, one deduces $\phi(\Mone^{\oplus r}) = r$ for every spectral point $\phi$ and consequently $\phi(T) \leq \uR(T^{\boxtimes N})^{1/N}$ for every $N$. Strassen proved that $\aR(T) = \sup\{ \phi(T):\phi \text{ spectral point}\}$.

In this context it is useful to introduce the {\it border subrank} of a tensor $T$, denoted $\uQ(T)$: this is the largest $r$ such that $\Mone^{\op r}$ is a degeneration of $T$. Similar to the asymptotic rank, there is a notion of \emph{asymptotic subrank}, defined by $\aQ(T) = \lim _{N \to \infty} \uQ(T^{\boxtimes N})^{1/N}$. Strassen proved that $\aQ(T) = \inf\{ \phi(T):\phi \text{ spectral point}\}$.

One can restrict the theory to subclasses of tensors which are closed under direct sum and Kronecker product, e.g., the subclasses of tight, oblique or free tensors. Conjecture \ref{strconj} is therefore equivalent to the conjecture that, given a tensor $T \in \bbC^m \otimes \bbC^m \otimes \bbC^m$, one has $\phi(T) \leq m$ for every spectral point.

In general, there is no systematic way to construct spectral points. Strassen defined a set of spectral points on the class of oblique tensors, that we discuss in \S \ref{section: strassen's functionals}. The first non-trivial examples of spectral points for all tensors of order three were determined only very recently, in \cite{2017arXiv170907851C}.

\subsection{Compressibility}
 A tensor $T\in A\ot B\ot C$ is $(\aaa',\bbb',\ccc')$-{\it compressible} (resp. $(\aaa',\bbb',\ccc')$-{\it incompressible}) if there exist (resp. do not exist)  linear spaces $A'\subset A^*$, $B'\subset B^*$, $C'\subset C^*$, respectively of dimensions $\aaa',\bbb',\ccc'$, such that $T|_{A'\ot B'\ot C'}=0$. The {\it total compressibility} of $T$ is the largest $\aaa'+\bbb'+\ccc'$ such that $T$ is $(\aaa',\bbb',\ccc')$-compressible.

A tensor is $\rho$-{\it multicompressible} if it is $(\aaa',\bbb',\ccc')$-compressible for all $\aaa',\bbb',\ccc'\in \bbN$ such that $\aaa'+\bbb'+\ccc'=\rho$.  

Incompressibility was introduced in \cite{MR3682743} and \cite{2016arXiv160807486L} as a measure of non-genericity. It was proved that a generic tensor in $A \otimes B \otimes C$ with $\bfa = \bfb = \bfc = m$ is   $(\bfa',\bfb',\bfc')$-incompressible if $m \leq \frac{(\aaa')^2+(\bbb')^2+(\ccc')^2+\aaa'\bbb'\ccc'}{\aaa'+\bbb'+\ccc'}$. In particular a generic tensor is not $(\sqrt{\frac{m}{3}},\sqrt{\frac{m}{3}},\sqrt{\frac{m}{3}})$-compressible, see \cite[Ex. 4.3]{2016arXiv160807486L} and consequently it is not $3\sqrt{\frac{m}{3}}$ multi-compressible.
 
It turns out that compressibility has a strong connection with Strassen's spectral theory and with the notion of slice rank.

\begin{definition}
Let $T \in A \otimes B \otimes C$. The {\it slice rank} of $T$, denoted $\slrk(T)$, is the smallest $r$ such that there exist three subspaces $A',B',C'$ of $A,B,C$ respectively such that $T \in A'\otimes B \otimes C + A \otimes B' \otimes C + A \otimes B \otimes C'$ and $\dim A' + \dim B' + \dim C' = r$. In other words, there exist vectors $a_i\in A$, $b_j \in B$, $c_k\in C$ such that $T= \sum_{i=1}^{r_A} a_i\ot X_i +\sum_{j=1}^{r_B} b_j\ot Y_j+ \sum_{k=1}^{r_C} c_k\ot Z_k$ with $r_A+r_B+r_C=r$ (and we suppress from the notation a reordering of the factors in the second and third summations).
\end{definition}

There is an equivalent definition in terms of compressibility: if $T \in A \otimes B \otimes C$ with $\dim A = \bfa, \dim B = \bfb, \dim C = \bfc$, then $\slrk(T) = \bfa + \bfb +\bfc - \kappa$ where $\kappa$ is the total compressibility of $T$.

For a tensor $T$, there is an asymptotic version of slice rank, defined as 
\[
\uwave{\slrk}(T) = \lim_{N \to \infty} \slrk(T^{\boxtimes N})^{1/N}. 
\]

Slice rank was introduced in \cite{Tao:SymmetricCrootLevPach} in connection with the cap set problem (see also \cite{tao,MR3631613}). In general, subrank is a lower bound for slice rank and consequently asymptotic subrank is a lower bound for asymptotic slice rank. In \cite{2017arXiv170907851C}, the authors showed that on the subclass of tight tensors, asymptotic slice rank and asymptotic subrank are equal and therefore asymptotic slice rank is entirely controlled by spectral points.

\subsection{Strassen's support functionals: minimal weighted average incompressibility}\label{section: strassen's functionals}
In this section, we describe the spectral points defined by Strassen \cite{Str:AsySpectrumTensors} on the class of tight tensors, highlighting the geometric perspective and the connections with compressibility.

A (increasing) complete flag in $\bbC^m$ is a sequence of subspaces $0 \subseteq V_1 \subseteq \cdots \subseteq V_{m} = \bbC^m$ with $\dim V_k = k$. Every increasing complete flag induces a unique (decreasing) complete flag in ${\bbC^m}^*$, defined by associating to a subspace $V_k \subseteq \bbC^m$ its annihilator $V_k^\perp \subseteq {\bbC^m}^*$.

Let $\cF$ denote set of triples of increasing complete flags in $A,B,C$. For $ f\in \cF$ with $f = (f_A,f_B,f_C)$, $f_A = ( 0 \subseteq A_1 \subseteq \cdots \subseteq A_\bfa = A)$ (and similarly for $f_B$ and $f_C$) and for every $T\in A\ot B\ot C$ define
\[
\tincompr_f(T):=\{ (i,j,k) \mid T|_{A_i ^\perp \ot B_j ^\perp\ot C_k^\perp}\neq 0\},
\]
the set of (indices of) incompressible subspaces in the triple  of flags $ f\in \cF$. 

Given a probability distribution $p$ on $[m]$, let $H(p) = - \sum_{j=1}^m p_j \log_2(p_j)$ denote its Shannon entropy, with the convention that $0\log_20=0$.

Let $p$ a probability distribution on $[\aaa]\times[\bbb]\times[\ccc]$, that is $p = (p_{ijk} : i \in [\bfa],j\in[\bfb],k\in [\bfc])$ with $p_{ijk} \in [0,1]$ and $\sum p_{ijk} = 1$. Let $p_A,p_B,p_C$ be the induced marginal distributions: more precisely, $p_A$ is the probability distribution on $[\bfa]$ defined by $p_{A,i} = \sum_{jk} p_{ijk}$ and similarly $p_B$ and $p_C$. 

Let $\theta = (\theta_A,\theta_B,\theta_C)$ be a probability distribution on $\{ 1,2,3\}$, that is a triple of numbers in $[0,1]$ with $\theta_A + \theta_B +\theta_C = 1$. In \cite{MR1089800}, for every $\theta$, Strassen defined a functional $\hat{\zeta}^\theta$, that he called \emph{support functional}, on the class of oblique tensors and characterized as follows: 
\begin{equation}\label{strzeta}
\begin{aligned}
&\log_2( \hat\zeta^{\th}(T)) = \\
&\min \Biggl\{ \max \left\{ \theta_A H(p_A) + \theta_B H(p_B) + \theta _C H(p_C) : \begin{array} {l}
										  p \text{ prob. dist. on} \\
                                                                                  \tincompr_f(T) 
                                                                                  \end{array}\right\} : \begin{array}{l} f = (f_A,f_B,f_C) \\ \text{triple of flags} \end{array}\Biggr\}.
\end{aligned}
\end{equation}

Strassen proved that if $T$ is tight, then $\aQ(T) = \inf _\theta \hat \zeta^{\th}(T)$. The original form of the asymptotic rank conjecture was stated in \cite{Strassen:AlgebraComplexity} in terms of the support functionals. It states that on the class of tight tensors the support functionals determine the entire spectrum:

\begin{conjecture} [\cite{Strassen:AlgebraComplexity}, Conj. 5.3] \label{conj: strassen functionals} The asymptotic spectrum of the class of tight tensors coincides with the set of support functionals.
\end{conjecture}

Note that if $T \in \bbC^m \otimes \bbC^m \otimes \bbC^m$, then the value of the support functionals on $T$ is bounded from above by $m$: indeed if $p$ is a probability distribution on $[m]$, then $H(p) \leq \log_2(m)$ with equality holding if and only if $p$ is the uniform distribution. In particular, if the support functionals define all spectral points as predicted by Conjecture \ref{conj: strassen functionals}, one obtains $\aR(T) = \sup_\theta \zeta^\theta(T) \leq 2^{\log_2(m)} = m$.

 \subsection{Compressibility of tight tensors}
The following result shows that tight tensors are highly compressible, compared to generic tensors. More precisely, a tight tensor is $(\bfa',\bfb',\bfc')$-compressible for $\bfa',\bfb',\bfc' \approx m/2 \gg \sqrt{\frac{m}{3}}$.

\begin{theorem} \label{compressthm}
Let $\bfa= \bfb=\bfc = m$ and let $T\in A \otimes B \otimes C$ be a tight 
tensor. Then $T$ is $(\lceil m/2\rceil, \lceil m/2\rceil, \lfloor 
m/2\rfloor)$-compressible and similarly permuting the order of the factors.
\end{theorem}

\begin{proof}
Let  $T$ be expressed in a tight basis and let $\tau_A,\tau_B,\tau_C : \{ 1 
\vvirg m\} \to \bbZ$ be the corresponding increasing injective functions, with 
$\tau_A + \tau_B + \tau_C$ identically $0$ on $\supp(T)$. We impose one 
additional normalization on $\tau_A,\tau_B,\tau_C$ as follows: we assume 
$\tau_A( \lfloor m/2 \rfloor) = \tau_B(\lceil m/2 \rceil) = -1$ and if 
$\tau_C(k) \geq 0$ then $\tau_C(j) \geq 2$: in order to do this, redefine 
$\tau_A = 3 \tau_A' - 3 \tau_A( \lfloor m/2 \rfloor) - 1$, $\tau_B' = \tau_B - 
\tau_B(\lceil m/2 \rceil) -1$ and $\tau_C' = \tau_C + \tau_A(\lfloor m/2 
\rfloor) + \tau_B(\lceil m/2 \rceil) + 2$. Notice that $\tau_A + \tau_B + \tau_C 
= 0$ if and only if $\tau_A' + \tau_B' + \tau_C' = 0$, so $\tau_A,\tau_B,\tau_C$ 
define the same tight support as $\tau_A' ,\tau_B' , \tau_C'$; moreover 
$\tau_A',\tau_B',\tau_C'$ are increasing, $\tau_A( \lfloor m/2 \rfloor) = 
\tau_B(\lceil m/2 \rceil) = -1$ 
and if $\tau_C(k) \geq 0$ then $\tau_C(k) \geq 2$ because $\tau_C \equiv 2 \mod 
3$. In fact $\tau_A,\tau_B,\tau_C \equiv 2 \mod 3$ and in particular they are 
never $0$. 

Now, we consider two cases:
 \begin{enumerate}
\item[(i)] if $\tau_C(\lceil m/2 \rceil) > 0$, then choose $A' = \langle 
\alpha_i : i \in \{ \lfloor m/2 \rfloor +1 \vvirg m \} \rangle$, $B' = \langle 
\beta_j : j \in \{ \lceil m/2 \rceil \vvirg m \} \rangle$ and $C' = \langle \{ 
\gamma_k : k \in \{ \lceil m/2 \rceil \vvirg m \} \rangle$. Notice that $\dim A' 
= \lceil m/2 \rceil$, $\dim B' = \lfloor m/2 \rfloor +1$ and $\dim C' = \lfloor 
m/2 \rfloor +1$; moreover, the sum of $\tau_A',\tau_B',\tau_C'$ on the product 
of these subsets is lower bounded by $\tau_A ( \lfloor m/2 \rfloor +1 ) + 
\tau_B(\lceil m/2 \rceil) + \tau_C(\lfloor m/2 \rfloor ) \geq 2 -1 + 2= 3 > 0$. 
This shows $T|_{A' \otimes B' \otimes C'} = 0$ because no elements of $\supp(T)$ 
appear in this range. In this case $T$ is $(  \lceil m/2 \rceil , \lfloor m/2 
\rfloor +1 , \lfloor m/2 \rfloor + 1)$-compressible, and in particular $(  
\lfloor m/2 \rfloor , \lceil m/2 \rceil , \lceil m/2 \rceil)$-compressible.

\item[(ii)] if $\tau_C(\lceil m/2 \rceil) < 0$, then choose $A' = \langle 
\alpha_i : i \in \{1 \vvirg \lfloor m/2 \rfloor \} \rangle $, $B' = \langle 
\beta_j : j \in \{1 \vvirg \lceil m/2  \rceil \} \rangle$ and $C' = \langle \{ 
\gamma_k : k \in \{1 \vvirg \lceil m/2\rceil \} \rangle$. Notice that $\dim A' = 
\lfloor m/2 \rfloor$, $\dim B' = \lceil m/2 \rceil$ and $\dim C' = \lceil m/2 
\rceil$; moreover, the sum of $\tau_A',\tau_B',\tau_C'$ on the product of these 
subsets is upper bounded by $\tau_A( \lfloor m/2 \rfloor ) + \tau_B( \lceil m/2 
\rceil) + \tau_C(\lceil m/2 \rceil) \leq -1 -1 -1 = -3$. This shows $T|_{A' 
\otimes B' \otimes C'} = 0$ because no elements of $\supp(T)$ appear in this 
range. In this case $T$ is $(  \lfloor m/2 \rfloor, \lceil m/2 \rceil, \lceil 
m/2 \rceil)$-compressible.
\end{enumerate}
\end{proof}

 We show that tight tensors with support equal to $\calS_{t\text{-}max,m}$ are 
highly multicompressible. More precisely, recall that generic tensors are not $3 
\sqrt{ \frac{m}{3}}$-multicompressible, whereas for tensors with support 
$\calS_{t\text{-}max,m}$ we have the following result.

\begin{proposition}\label{prop: multicompressible tight}
Let $\bfa = \bfb = \bfc = m$ and let $T\in A \otimes B \otimes C$ be a tight 
tensor with support $\cS_{t\text{-}max,m}$. Then $T$ is $(3\lfloor m/2 
\rfloor+1)$-multicompressible.
\end{proposition}
\begin{proof}
Recall from Example \ref{example: max tight support}, $\tau_A ,\tau_B, \tau_C : 
[m] \to \bbZ$, with $\tau_A (i) = \tau_B(i) = \tau_C(i) = i - \ell$ if $m = 
2\ell+1$ is odd and with $\tau_A(i) =.i-\ell+1$, $\tau_B(j) = \tau_C(j) = 
j-\ell$ if $m = 2\ell$ is even.

Fix $(\bfa',\bfb',\bfc')$ with $\bfa' + \bfb' + \bfc' = 3\lfloor m/2 \rfloor+1 = 
3 \ell +1$. We determine $A' \subseteq A^*,B' \subseteq B^*,C' \subseteq C^*$ 
with $\dim A' = \bfa'$, $\dim B' = \bfb'$, $\dim C' = \bfc'$ such that $T|_{A' 
\otimes B'\otimes C'} = 0$. Let $A' = \langle \alpha^i : i \in [\bfa'] \rangle$, 
$B' = \langle \beta^j : j \in [\bfb'] \rangle$ ,$C' = \langle \gamma^k : k \in 
[\bfc'] \rangle$.

We claim that $T|_{A' \otimes B'\otimes C'} = 0$. This follows from the fact 
that $[\bfa'] \times [\bfb']\times [\bfc'] \cap \calS_{t\text{-}max,m} = 
\emptyset$. If $(i,j,k) \in [\bfa'] \times [\bfb']\times [\bfc']$, we have 
$\tau_A(i) + \tau_B(j) + \tau_C(k) \leq i - \ell +1 + j - \ell + k - \ell \leq 
\bfa'-1 + \bfb'-1 + \bfc'-1 - 3\ell \leq 3\ell + 1 - 2 - 3\ell = -1$ (here the 
first inequality is in fact an equality if $m$ is even). In particular, there 
are no elements $(i,j,k) \in [\bfa'] \times [\bfb']\times [\bfc']$ such that 
$\tau_A(i)+\tau_B(j) + \tau_C(k) = 0$.
\end{proof}

However, we observe that highly multicompressible tensors are not necessarily 
tight:
\begin{example}
This is an example of a $3 \lfloor m/2 \rfloor$-multicompressible tensor that is 
not tight. Let 
\begin{align*}
T = &a_0\otimes b_0\otimes c_0+a_1\otimes b_1\otimes c_1+a_2\otimes b_2\otimes 
c_2+a_3\otimes b_3\otimes c_3+ \\ 
&+(a_0+a_1+a_2+a_3)\otimes (b_0+b_1)\otimes (c_2+c_3)+ (a_1+a_2+a_3)\otimes 
b_2\otimes (c_2+c_3)+ \\ 
&+ (a_1 +a_2+a_3)\otimes b_3\otimes c_3+(a_2+a_3)\otimes b_3\otimes c_2
\end{align*}
be a tensor in $A \otimes B \otimes C$ with $\bfa = \bfb = \bfc = 4$. A direct 
calculation shows that $T$ has trivial annihilator $\frakg_T$, therefore it is 
not tight. It is easy to verify that $T$ is $6$-multicompressible. 

Taking direct sums of copies of the tensor above one obtains highly 
compressible, not tight tensors in higher dimensions.
\end{example}

\subsection{Additional remarks on compressibility in general} 
One can discuss a restricted form of multicompressibility, by letting only the 
dimensions of two factors vary. In this context we have the following result:

\begin{proposition} Let $T \in A \otimes B \otimes C$ with $\bfa = \bfb = \bfc= 
m$. For every $\bfb' ,\bfc'$ with $\bfb' + \bfc' \leq m - \lceil 
\sqrt{m-1}\rceil$, $T$ is $(1,\bfb',\bfc')$-compressible and similarly permuting 
the roles of the three factors.
\end{proposition}
\begin{proof}
The result is immediate if $T$ is not concise, as $T$ would have an $m\times m$ block of zeros,
and so, up to permuting factors, it would be $(1,\bfb',\bfc')$-compressible for $\bfb'+\bfc'\leq m$. 

Suppose $T$ is concise. Let 
$\sigma_r(\bbP B \times \bbP C) \subseteq \bbP (B \otimes C)$ denote the 
subvariety of rank at most $r$ elements in $\bbP (B \otimes C)$. We have $\dim 
(\sigma_r(Seg(\BP B\times \BP C)))=2rm-r^2-1$. By conciseness, the image of the 
flattening map $T_A : A^* \to B \otimes C$ has dimension $m$, so its 
projectivization $\bbP (T_A(A^*))$ intersects $\sigma_{r}(Seg(\BP B\times \BP 
C))$ when $r = m-\lceil \sqrt{m-1}\rceil$. 

So fix $r =  m-\lceil \sqrt{m-1}\rceil$ and let $\alpha \in A^*$ such that 
$[T_A(\alpha)] \in \bbP (T_A(A^*)) \cap \sigma_{r}(Seg(\BP B\times \BP C))$. 
Then $T_A(\alpha)$ has rank at most $r=m-\lceil \sqrt{m-1}\rceil$.
 
 Choose bases such that $T_A(\alpha) = b_1\ot c_1+\cdots + b_r\ot c_r$ and let 
$A' = \langle \alpha \rangle$,  $B'= \langle \beta_1 \vvirg \beta_{\bfb'} 
\rangle$
 and $C'=\langle \gamma_{\bfb' + 1} \vvirg \gamma_{\bfb'+\bfc'} \rangle$. Then 
$\dim A' = 1$, $\dim B' = \bfb'$ and $\dim C' = \bfc'$; we have $T|_{A' \otimes 
B'\otimes C'} = 0$, so $T$ is $(1,\bfb',\bfc')$-compressible.
 \end{proof}

More generally, we show that maximally compressible tensors (in the sense of 
\cite{MR3682743}) are also highly multicompressible.

\begin{proposition}\label{prop: maxcompr implies 2m-1 multicompr}
 Let $\bfa = \bfb = \bfc = m$ and let $T \in A \otimes B \otimes C$ be 
$(m-1,m-1,m-1)$-compressible. Then $T$ is $2m-1$-multicompressible.
 \end{proposition}
\begin{proof}
After fixing bases in $A,B,C$, we may assume without loss of generality that 
$T|_{a_0^\perp \otimes b_0^\perp \otimes c_0^\perp} = 0$; in particular $T$ can 
be written as $T = a_0 \otimes M_A + b_0 \otimes M_B + c_0 \otimes M_C$ where 
$M_A \in B \otimes C$ and similarly $M_B,M_C$ (and reordering the factors in the 
second and third summand). Let $\bfa' , \bfb' , \bfc' \leq m$ with $\bfa' + 
\bfb' + \bfc'  = 2m-1$. Moreover, we may assume $M_A \in b_0^\perp \otimes 
c_0^\perp$ because expressing $T$ in the fixed basis, we can include summands 
including $b_0,c_0$ in $b_0 \otimes M_B + c_0 \otimes M_C$ and in fact we may 
fix bases so that $M_A = \sum_1^r b_i \otimes c_i$ for some $r \leq m-1$.

If $\bfa' , \bfb' , \bfc' < m$, let $A' \subseteq a_0^\perp, B' \subseteq 
b_0^\perp, C' \subseteq c_0^\perp$, so $T|_{A' \otimes B'\otimes C'} = 0$.

Suppose $\bfa' =m$; therefore $A' = A^*$ and $\bfb' + \bfc' = m-1$. Let $B' = 
\langle \beta^1 \vvirg \beta ^{\bfb'} \rangle$ and $C' = \langle c^{\bfb'+1} 
\vvirg c^{m-1} \rangle$. Then $T|_{A' \otimes B' \otimes C'} = 0$. This 
concludes the proof.
\end{proof}

\begin{remark}
Proposition \ref{prop: maxcompr implies 2m-1 multicompr} implies that the Coppersmith-Winograd tensors $T_{cw,q}$ and $T_{CW,q}$, introduced in 
\cite{CopperWinog:MatrixMultiplicationArithmeticProgressions} and used in the research of upper bounds for the complexity of matrix multiplication \cite{stothers,WilliamsFasterThanCW,LeGall:2014:PTF:2608628.2608664} are respectively 
$(2q+1)$-multicompressible and $(2q+3)$-multicompressible. 
\end{remark}

\begin{remark}
Let $\bfa = \bfb= \bfc = \bfn^2$ and consider $M_{\langle {\bf n}\rangle}\in A 
\otimes B \otimes C$. Then $M_{\langle {\bf n}\rangle}\in A \otimes B \otimes C$ 
is $3\lfloor {\bf n}^2/2\rfloor$-multicompressible. The proof is similar to that 
of Proposition \ref{prop: maxcompr implies 2m-1 multicompr}. After a change of 
basis the flattening map $M_{\langle {\bf n}\rangle}: A^* \to B \otimes C$ can 
be written as a $(\bfn \times \bfn)$-block diagonal matrix of linear forms on 
$A$, whose diagonal blocks are all equal to the matrix $(\alpha^i_j)$, see 
\cite[Exercise 2.1.7.4]{MR3729273}. In this form, it is easy to see that 
$\Mamu{\bfn}$ is $(\bfn^2 , \bfb' , \bfc')$-compressible for every 
$(\bfb',\bfc')$ with $\bfb'+ \bfc' = \bfn^2$.

At this point, consider $\bfa', \bfb' , \bfc'$ with $\bfa'+ \bfb'+\bfc' = 
3\lfloor {\bf n}^2/2\rfloor$. Notice that $(\bfa' + \bfb')+(\bfb' + 
\bfc')+(\bfa'+\bfc') \leq 3{\bf n}^2$, so at least one among $(\bfa' + \bfb'), 
(\bfb' + \bfc'),(\bfa'+\bfc')$ is bounded from above by $\bfn^2$. Suppose $\bfb' 
+ \bfc' \leq \bfn^2$. From the argument above $\Mamu{\bfn}$ is $(\bfn^2, 
\bfb',\bfc')$-compressible and therefore $(\bfa',\bfb',\bfc')$-compressible.
\end{remark}

Finally,    when $\bfa = \bfb = \bfc = 
m$, every $T \in A \otimes B \otimes C$ with border rank $r$ is $(3m - 
r)$-multicompressible \cite{2016arXiv160807486L}. For instance, the tensor $T_{std,m}$ defined in the proof 
of Proposition \ref{prop: secants and tight} is $(2m-1)$-multicompressible.

\subsection{Combinatorial geometry of tight sets and compressibility}
This subsection discusses a  combinatorial approach towards proving compressibility of tight tensors.
Theorem  \ref{compressthm} and Proposition \ref{prop: multicompressible tight}
may be recovered from Proposition \ref{ctightprop} below.

The three functions $\tau_A,\tau_B,\tau_C : [m] \to \bbZ$ define a line 
arrangement in $\bbR^2$ as follows. Consider in $\bbR^3$ (with coordinates 
$(x,y,z)$) the following union of planes: 

\begin{equation}\label{eqn: plane arrangement} 
\hat{\calA} =  \bigsqcup_{i =0}^{m-1} \{ x = \tau_A(i)\} \cup \bigsqcup 
_{j=0}^{m-1} \{ y = \tau_B(j)\} \cup \bigsqcup_{k =0}^{m-1} \{ z = \tau_C(k)\}.
\end{equation}
Let $\calA$ be the intersection of $\hat{\calA}$ with the plane $\Pi = \{ x+y+z 
= 0\}\subset \mathbb R^3$. The set $\calA$ is an arrangement of three families 
of parallel lines, each consisting of $m$ lines: with respect to coordinates 
$x,y$ in $\Pi$, the three families of lines are $\calA = \bigsqcup_{i =0}^{m-1} 
\{ x = \tau_A(i)\}\cup \bigsqcup _{j=0}^{m-1} \{ y = \tau_B(j)\} \cup 
\bigsqcup_{k =0}^{m-1} \{ x+y = -\tau_C(k)\}$; we say that the 
$x$-\emph{direction} of $\calA$ is the union of the lines with constant $x$, the 
$y$-direction is the union of lines with constant $y$ and the 
$z$-\emph{direction} is the union of lines with slope $-1$. A subset of lines 
$\calA' \subseteq \calA$ is called a sub-arrangement if it contains at least one 
line in each direction.

The set $\{ p \in \Pi : p \text{ belongs to exactly two lines in $\calA$}\}$ is 
called the \emph{set of double intersection points} of $\calA$. The set 
$\frakJ(\calA) := \{ p \in \Pi : p \text{ belongs to three lines in $\calA$}\}$ 
is the set of \emph{joints} in $\calA$.

\begin{figure}[!htb]
\begin{center}
\begin{tikzpicture}
 \draw[very thick,blue] (-1,0) -- (8,0);
 \draw[very thick,blue] (-1,.5) -- (8,.5);
 \draw[very thick,blue] (-1,2) -- (8,2);
 \draw[very thick,blue] (-1,3) -- (8,3);
 \draw[very thick,blue] (-1,4.5) -- (8,4.5);
 \draw[very thick,blue] (-1,6) -- (8,6);

 \draw[very thick,red] (0,-1) -- (0,7);
 \draw[very thick,red] (1,-1) -- (1,7);
 \draw[very thick,red] (3,-1) -- (3,7);
 \draw[very thick,red] (4,-1) -- (4,7);
 \draw[very thick,red] (6.5,-1) -- (6.5,7);

 \draw[very thick,green] (-1,3) -- (3,-1);
 \draw[very thick,green] (-1,4) -- (4,-1);
 \draw[very thick,green] (-1,5.5) -- (5.5,-1);
 \draw[very thick,green] (4,7) -- (8,3);
 \draw[very thick,green] (0,7) -- (8,-1);
 
\node at (0,2) [circle,fill,inner sep=2pt]{};
\node at (0,3) [circle,fill,inner sep=2pt]{};
\node at (0,4.5) [circle,fill,inner sep=2pt]{};
\node at (1,2) [circle,fill,inner sep=2pt]{};
\node at (1,6) [circle,fill,inner sep=2pt]{};
\node at (3,0) [circle,fill,inner sep=2pt]{};
\node at (4,0.5) [circle,fill,inner sep=2pt]{};
\node at (4,3) [circle,fill,inner sep=2pt]{};
\node at (6.5,.5) [circle,fill,inner sep=2pt]{};
\node at (6.5,4.5) [circle,fill,inner sep=2pt]{};
\end{tikzpicture}
\caption{An arrangement of lines on the plane $\Pi$: the red lines are in the  $x$-direction, the blue lines in the $y$-direction and the green lines in the  $z$-direction. The joints are marked with black dots.}\label{figure: arrangement}
\end{center}
\end{figure}
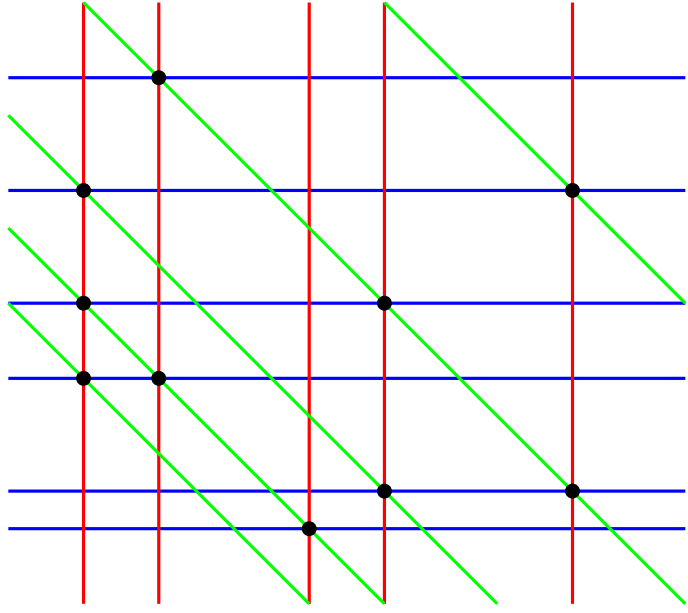

\begin{lemma}\label{Supportandjoints}
Let $T$ be a tight tensor in a tight basis and let $\calA$ be the corresponding 
arrangement of lines. Then $(\tau_A,\tau_B,\tau_C) : [m]^{\times 3} \to \bbR^3$ 
maps $\supp(T)$ bijectively to a subset of $\frakJ(\mathcal A)$. In particular, 
if $\calA$ is an arrangement with $\mathfrak{J}(A)=\emptyset$ then $T= 0$.
\end{lemma}
\begin{proof}
If $(i,j,k) \in \supp(T)$ then $\tau_A(i)+\tau_B(j)+\tau_C(k) = 0$, so 
$(\tau_A(i), \tau_B(j) , \tau_C(k)) \in \Pi$ is a point of $\calA$ lying on 
three lines. In particular $(\tau_A(i), \tau_B(j) , \tau_C(k)) \in 
\frakJ(\calA)$.
\end{proof}

If $T$ is a tight tensor in a tight basis and $\calA$ is the corresponding line 
arrangement with $\supp(T) \subseteq \frakJ(\calA)$, we say that $T$ is 
\emph{supported} on $\calA$. Properties of the support of a tight tensor in a 
tight basis can be translated into geometric and combinatorial properties of 
$\calA$. For instance, compressibility in given coordinates can be studied 
combinatorially as follows.

\begin{proposition}\label{ctightprop}
Let $T$ be a tight tensor and let $\calA$ be the corresponding line arrangement. 
If there exists a sub-arrangement $\calA'$ of $\calA$ consisting of $\bfa'$ 
lines in the $x$ direction, $\bfb'$ lines in the $y$ direction, and $\bfc'$ 
lines in the $z$ direction with $\mathfrak{J}(\calA')=\emptyset$, then $T$ is 
$(\bfa' , \bfb' , \bfc')$-compressible. 
\end{proposition}
\begin{proof}
After the identification of $A,B,C$ with their duals determined by the choice of 
bases, let $A' \subseteq A^*$ be the subspace spanned by the basis elements $\{ 
\alpha^i : \{x = \tau_A(i)\} \in \calA'\}$ and similarly $B'$ and $C'$. Then $T' 
= T|_{A' \otimes B' \otimes C'}$ is tight, and the corresponding arrangement is 
$\calA'$. Since $\frakJ(\calA') = \emptyset$, we conclude by Lemma 
\ref{Supportandjoints}.
\end{proof}

\section{Propagation of symmetries}\label{section: propagation of symmetries}

 Recall that $\Phi : GL(A) \times GL(B) \times GL(C) \to GL(A \otimes B \otimes 
C)$ defines the natural action of $GL(A) \times GL(B) \times GL(C)$ on $A 
\otimes B \otimes C$ and $G$ is the image of $\Phi$ in $GL(A \otimes B \otimes 
C)$. Denote by $G_T$ the stabilizer of a tensor $T$ in $G$ and by $\frakg_T$ the 
Lie algebra of $G_T$, namely the annihilator of $T$ under the Lie algebra action 
of $\frakg = \frakgl(A) \oplus \frakgl(B) \oplus \frakgl(C) / \frakz_{A,B,C}$. 
 
 We have the following result on propagation of symmetries.

\begin{theorem}\label{thm: propagation of symmetries}
 Let $T \in A_1 \otimes B_1 \otimes C_1$ and $S \in A_2 \otimes B_2 \otimes C_2$ 
be concise tensors. Then
 \begin{enumerate}[(i)]
\item as subalgebras of $\left( \frakgl(A_1 \oplus A_2) \oplus \frakgl(B_1 
\oplus B_2) \oplus \frakgl(C_1 \oplus C_2) \right) / \frakz_{A_1\oplus A_2,B_1 
\oplus B_2, C_1\oplus C_2}$, 
\[
\frakg_{T \oplus S} = \frakg_{T} \oplus \frakg_{S};
\]
\item as subalgebras of $\left( \frakgl(A_1 \otimes A_2) \oplus \frakgl(B_1 
\otimes B_2) \oplus \frakgl(C_1 \otimes C_2) \right) / \frakz_{A_1 \otimes A_2, 
B_1 \otimes B_2 , C_1 \otimes C_2}$,
\[
\frakg_{T \boxtimes S} \supseteq \frakg_{T} \otimes \Id_{A_2 \otimes B_2 \otimes 
C_2} + \Id_{A_1 \otimes B_1 \otimes C_1} \otimes \frakg_{S};
\]
\item if $\frakg_{T} = 0$ and $\frakg_{S} = 0$ then $\frakg_{T \boxtimes S} = 
0$.
 \end{enumerate}
\end{theorem}

The containment of (ii) in Theorem \ref{thm: propagation of symmetries} can be strict, for instance in the case of the matrix multiplication tensor. Additional examples are provided in \cite{CGLVkron}. We propose the following problem, which addresses the general study of propagation of non-genericity properties under Kronecker powers, in the spirit of Strassen's asymptotic rank conjecture and its generalizations.

\begin{problem}
 Characterize tensors $T \in A \otimes B \otimes C$ such that $\frakg_T \otimes 
\Id_{A \otimes B \otimes C} + \Id_{A \otimes B \otimes C} \otimes \frakg_T$ is 
strictly contained in $\frakg_{T^{\boxtimes 2}} \in A^{\otimes 2} \otimes 
B^{\otimes 2} \otimes C^{\otimes 2}$.
\end{problem}

\begin{proof}[Proof of Theorem. \ref{thm: propagation of symmetries}]
Throughout the proof, we use the summation convention for which repeated upper 
and lower indices are to be summed over. The range of the indices is omitted as 
it should be clear from the context.

\subsection*{Proof of  (i)}
Let $T \in A_1 \otimes B_1 \otimes C_1$ and $S \in A_2 \otimes B_2 \otimes C_2$. 
Fix bases of $A_1,B_1,C_1,A_2,B_2,C_2$ and write $T  = T ^{i_1j_1k_1} 
a_{i_1}^{(1)} \otimes b_{j_1}^{(1)} \otimes c_{k_1}^{(1)}$ and $S = 
S^{i_2j_2k_2} a_{i_2}^{(2)} \otimes b_{j_2}^{(2)} \otimes c_{k_2}^{(2)}$. Let $L 
= (U,V,W) \in \frakgl(A_1 \oplus A_2) \oplus \frakgl(B_1 \oplus B_2) \oplus 
\frakgl(C_1 \oplus C_2)$. We want to prove that if $L. (T \oplus S) = 0$, then 
for $\ell = 1,2$, there is $L_\ell \in \frakgl(A_\ell) \oplus \frakgl(B_\ell) 
\oplus \frakgl(C_\ell)$ such that $L = L_1 + L_2$ with $L_1. T = 0$ and $L_2.S = 
0$. Write $X = X_{11} + X_{12} + X_{21} + X_{22}$ where $X_{11} \in 
\Hom(A_1,A_1)$ and similarly for the other summands. Consider $X_{21} (T 
\oplus S) = X_{21} (T)$: this is an element of $A_2 \otimes B_1 \otimes C_1$. 
No other summand of $X$, nor $Y$ or $Z$ generate a nonzero component in this 
space. Therefore, $X_{21} (T) = 0$ and by conciseness we deduce $X_{21} = 0$. 
Similarly $X_{12} = 0$ so that 
$X = X_{11} + X_{22} \in \frakgl(A_1) \oplus \frakgl(A_2)$ and similarly for $Y$ 
and $Z$. For $\ell = 1,2$, let $L_\ell = (X_{\ell\ell} , Y_{\ell\ell} , 
Z_{\ell\ell})$. Then $L = L_1 + L_2$ and $L.(T \oplus S) = L_1 . T + L_2 .S$; 
notice $L_1.T \in A_1 \otimes B_1 \otimes C_1$ and $L_2. S \in A_2 \otimes B_2 
\otimes C_2$ are linearly independent, so if $L.(T\oplus S)=0$, we have $L_1 
\in \frakg_{T}$ and $L_2 \in \frakg_{S}$.  

\subsection*{Proof of  (ii)}
This is a straightforward consequence of the Leibniz rule. In general if 
$\frakg_1$ acts on a space $V_1$ and $\frakg_2$ acts on a space $V_2$, then the 
action of $\frakg_1 \oplus \frakg_2$ on $V_1 \otimes V_2$ is given by the 
Leibniz rule via $(L_1,L_2)\mapsto L_1 \otimes \Id_{V_2} + \Id_{V_1} \otimes 
L_2$. If $v_1 \in V_1$ is annihilated by $\frakg_1$, and $v_2 \in V_2$ is 
annihilated by $v_2$, then $\frakg_1 \oplus \frakg_2$ annihilates $v_1 \otimes 
v_2$ via the induced action.

\subsection*{Proof of  (iii)}
Fix bases in all spaces. Let $L = (U,V,W) \in \frakgl(A_1 \otimes A_2) \otimes 
\frakgl(B_1 \otimes B_2)  \otimes \frakgl(C_1 \otimes C_2)$ and write $U$ as an 
$\bfa_1\bfa_2\times \bfa_1\bfa_2$  matrix $u^{i_1' i_2'}_{i_1i_2}$, and 
similarly for $V$ and $W$. Our goal is to prove that if $L. (T_1 \boxtimes T_2)= 
0$, then $L \in \frakz_{A_1\otimes A_2, B_1 \otimes B_2, C_1 \otimes C_2}$.

Write $T_1= T^{ijk} a_i^{(1)} \otimes b_j^{(1)} \otimes c_k^{(1)}$ and $T_2 = 
S^{i'j'k'} a_{i'}^{(2)} \otimes b_{j'}^{(2)} \otimes c_{k'}^{(2)}$. The 
equations for the symmetry Lie algebra $\frakg_{T_1 \boxtimes T_2}$ is $L.(T_1 
\boxtimes T_2) = 0$; in coordinates, for every $i_1,i_2,j_1,j_2,k_1,k_2$, we 
have (using the summation convention)
\begin{equation}\label{killbox}
u^{i_1i_2}_{i_1'i_2'} T ^{i_1' j_1k_1} S^{i_2' j_2 k_2} + v^{j_1j_2}_{j_1'j_2'} 
T ^{i_1 j_1'k_1} S^{i_2 j_2' k_2} + w^{k_1k_2}_{k_1'k_2'} T ^{i_1 j_1k_1'} 
S^{i_2 j_2 k_2'}  = 0 
 \end{equation}
 Let $U(i_1j_1k_1)\in \Hom(A_2 , A_2)$ be the matrix whose $(i_2,i_2')$-th entry 
is $u^{i_1 i_2}_{i_1'i_2'}T^{i_1' j_1k_1}$ and similarly $V(i_1j_1k_1)$ and 
$W(i_1j_1k_1)$. Let $L(i_1j_1k_1) = (U(i_1j_1k_1),V(i_1j_1k_1),W(i_1j_1k_1))$. 
From \eqref{killbox}, we have $L(i_1j_1k_1). T_2 = 0$, namely $L(i_1j_1k_1) = 
\frakz_{A_2, B_2,C_2}$. Therefore, one has $U(i_1j_1k_1) = u(i_1j_1k_1)\Id_{A_2}, 
V(i_1j_1k_1) = v(i_1j_1k_1)\Id_{B_2},W(i_1j_1k_1) = w(i_1j_1k_1)\Id_{C_2}$, with 
$u(i_1j_1k_1) + v(i_1j_1k_1) + w(i_1j_1k_1) = 0$.
 
 In particular, if $i_2\neq i_2'$, we have $U(i_1j_1k_1)^{i_2}_{i_2'} = 0$ for 
every $i_1,j_1,k_1$, which by definition provides $u^{i_1i_2}_{i_1'i_2'} 
T^{i_1'j_1k_1} = 0$. This implies that the $\bfa_1 \times \bfa_1$ matrix 
$(u^{i_1i_2}_{i_1'i_2'}) _{i_2,i_2'}$ satisfies $((u^{i_1i_2}_{i_1'i_2'}) 
_{i_2,i_2'},0,0) \in \frakg_{T_1}$. By conciseness, this implies 
$(u^{i_1i_2}_{i_1'i_2'}) _{i_2,i_2'} = 0$, and therefore $u^{i_1 i_2}_{i_1'i_2'} 
= 0$ for every $i_1,i_1'$ and every $i_2 \neq i_2'$. By exchanging the role of 
the two tensors, we obtain that $u^{i_1 i_2}_{i_1'i_2'} = 0$ for every $i_1 \neq 
i_1'$ and every $i_2 ,i_2'$. We deduce that $U$ is diagonal. Similar argument 
applies to $V$ and $W$.
 
Then, for each fixed $i_1,i_2,j_1,j_2,k_1,k_2$, \eqref{killbox} reduces 
to (with no summation) $T^{i_1j_1k_1}S^{i_2j_2k_2} 
(u^{i_1i_2}_{i_2i_2}+v^{j_1j_2}_{j_1j_2}+w^{k_1k_2}_{k_1k_2})=0$. Since our 
choice of bases is arbitrary, we may assume that $T^{i_1j_1k_1} \neq 0 \neq 
S^{i_2j_2k_2}$. Then taking  different values of $k_1,k_2$ and fixing 
$i_1,i_2,j_1,j_2$, we see all the $w^{k_1k_2}_{k_1k_2}$ must  be equal and 
similarly for $U$ and $V$. This shows that $U = \lambda \Id_{A_1 \otimes A_2}$, 
$V = \mu \Id_{B_1 \otimes B_2}$ and $W = \Id_{C_1 \otimes C_2}$. By evaluating 
\eqref{killbox} one last time, we see $\lambda = \mu = \nu$ that is $L \in 
\frakz_{A_1\otimes A_2, B_1 \otimes B_2, C_1 \otimes C_2}$.
 \end{proof}
 
\bibliographystyle{amsalpha}
\bibliography{bibTight}

\end{document}